\documentclass[11pt,a4paper,reqno]{amsart}

\usepackage[english]{babel}
\usepackage[T1]{fontenc}
\usepackage[utf8]{inputenc}
\usepackage{geometry}
\usepackage{amsmath,amssymb,amsthm,mathrsfs}
\usepackage{enumitem}
\usepackage[hidelinks]{hyperref}

\numberwithin{equation}{section}

\newtheorem{theorem}{Theorem}[section]
\newtheorem{proposition}[theorem]{Proposition}

\newtheorem{corollary}[theorem]{Corollary}
\newtheorem{thmA}{Theorem}

\theoremstyle{definition}
\newtheorem{definition}[theorem]{Definition}
\newtheorem{example}[theorem]{Example}
\newtheorem{remark}[theorem]{Remark}

\newcommand{\kk}{\Bbbk}
\newcommand{\Z}{\mathbb{Z}}
\newcommand{\Ecal}{\mathcal{E}}
\DeclareMathOperator{\Hom}{Hom}
\DeclareMathOperator{\Alt}{Alt}
\DeclareMathOperator{\Bich}{Bich}
\DeclareMathOperator{\SymBich}{SymBich}
\DeclareMathOperator{\AltBich}{AltBich}
\DeclareMathOperator{\Quad}{Quad}
\DeclareMathOperator{\Ind}{Ind}
\DeclareMathOperator{\Aut}{Aut}
\DeclareMathOperator{\Opext}{Opext}
\DeclareMathOperator{\id}{id}

\title[Cocentral Split Abelian Hopf Extensions]
{Cocentral Split Abelian Hopf Algebra Extensions from Crossed Cocycles}
\author{C\'esar Galindo and Giovanny Mora}
\date{\today}
\hypersetup{%
  pdftitle={Cocentral Split Abelian Hopf Algebra Extensions from Crossed Cocycles},
  pdfauthor={Cesar Galindo and Giovanny Mora},
  pdfsubject={Cocentral algebra-split abelian Hopf algebra extensions and obstruction theory},
  pdfkeywords={Hopf algebra extensions, algebra-split abelian extensions, group cohomology, Coxeter groups}
}

\begin{document}

\begin{abstract}
We study cocentral algebra-split abelian Hopf algebra extensions over an
algebraically closed field of characteristic zero for a fixed action of a
group on a finite abelian group. We describe the extension classes in terms of
crossed families of normalized \(2\)-cocycles and construct the
obstruction to representing crossed families of cohomology classes by
cocycles satisfying the crossed identity. A bicharacter obstruction maps to
the strict lifting obstruction and may remain nonzero even when its image
vanishes. For permutation modules, we obtain an explicit description of the
corresponding cocentral extension groups. We also study finite reductions of
geometric representations of Coxeter groups and compute the first cohomology
of the associated linear coefficient modules for finite dihedral groups and
the infinite dihedral group.
\end{abstract}

\maketitle


\section{Introduction}

\subsection*{Cocentral algebra-split extensions}

Extensions of Hopf algebras provide a standard method for constructing
new Hopf algebras from known ones. When the kernel is commutative and
the quotient is cocommutative, they form the abelian extension theory
developed through the Kac exact sequence and subsequent cohomological
formulations~\cite{Singer1972,Kac1968,Hofstetter1994,
Masuoka2002Extensions,Schauenburg2002Kac,GrunenfelderMastnak2004}.
Once the action and coaction are fixed, the corresponding extension
classes form an \(\Opext\) group~\cite{Schauenburg2002Kac}.

Throughout, \(\kk\) is an algebraically closed field of characteristic
zero. Let \(A\) be a finite abelian group and let \(\Gamma\) act on
\(A\). We study cocentral abelian extensions with kernel \(\kk^A\),
quotient \(\kk\Gamma\), and fixed action. Cocentrality makes the
coaction trivial, while algebra-splitness allows the algebra cocycle to
be chosen trivial. The resulting algebra is
\(\kk^A\rtimes\kk\Gamma\), and its coproduct is determined by a family
\(\tau=(\tau_x)_{x\in\Gamma}\) of normalized group \(2\)-cocycles on
\(A\) satisfying
\[
\tau_{xy}=\tau_x(x\cdot\tau_y).
\]
We call such a family a strict Hopf datum; the term ``strict'' records
that actual cocycle representatives, rather than only their cohomology
classes, satisfy the crossed identity.

Let \(C^1(A,\kk^\times)\) be the group of normalized functions on \(A\),
let
\[
(\partial f)(\alpha,\beta)
=
f(\alpha)f(\beta)f(\alpha+\beta)^{-1},
\qquad
B^2(A,\kk^\times)=\partial C^1(A,\kk^\times),
\]
and apply \(\partial\) pointwise to \(\Gamma\)-cochains. Changes of
homogeneous section give the quotient
\[
\Ecal(A,\Gamma)
=
\frac{Z^1(\Gamma,Z^2(A,\kk^\times))}
     {\partial Z^1(\Gamma,C^1(A,\kk^\times))}.
\]

\begin{thmA}\label{thmA:general-obstruction}
For fixed \(A\), \(\Gamma\), and action, \(\Ecal(A,\Gamma)\) classifies
cocentral algebra-split abelian extensions with kernel \(\kk^A\) and quotient
\(\kk\Gamma\), up to equivalence inducing the identity on both the kernel
and the quotient. Moreover, there is an exact sequence
\[
1
\longrightarrow
\frac{Z^1(\Gamma,B^2(A,\kk^\times))}
     {\partial Z^1(\Gamma,C^1(A,\kk^\times))}
\longrightarrow
\Ecal(A,\Gamma)
\longrightarrow
Z^1(\Gamma,H^2(A,\kk^\times))
\xrightarrow{\delta}
H^2(\Gamma,B^2(A,\kk^\times)).
\]
Thus a cohomological datum \(d\) is represented by a strict Hopf datum if
and only if \(\delta(d)\) is trivial.
\end{thmA}

Propositions~\ref{prop:tau-conditions}
and~\ref{prop:equivalence-sections} give the classification in
Theorem~\ref{thmA:general-obstruction}. Theorem~\ref{thm:general-obstruction}
and Corollary~\ref{cor:general-obstruction-exact} give its obstruction
and exact sequence.

\subsection*{Bicharacter and strict obstructions}

For finite abelian \(A\), the commutator map gives a natural
isomorphism \(H^2(A,\kk^\times)\cong\AltBich(A)\), where
\(\AltBich(A)\) denotes the group of alternating bicharacters on
\(A\)~\cite[Proposition~2.6]{Tambara2000}. If
\(\Bich(A)\) and \(\SymBich(A)\) denote the groups of bicharacters and
symmetric bicharacters, respectively, then
\[
1
\longrightarrow
\SymBich(A)
\longrightarrow
\Bich(A)
\xrightarrow{\Alt}
\AltBich(A)
\longrightarrow
1.
\]
Here
\[
\Alt(b)(\alpha,\beta)
=
\frac{b(\alpha,\beta)}{b(\beta,\alpha)}.
\]
We write
\[
\vartheta:\AltBich(A)\xrightarrow{\sim}H^2(A,\kk^\times)
\]
for the \(\Gamma\)-equivariant isomorphism characterized by
\(\vartheta(\Alt(b))=[b]\).

For \(d\in Z^1(\Gamma,\AltBich(A))\), choose bicharacters \(b_x\) with
\(b_1=1\) and \(\Alt(b_x)=d_x\). Their defect
\[
\Omega^{\operatorname{bich}}_d(x,y)
=
b_x(x\cdot b_y)b_{xy}^{-1}
\]
takes values in \(\SymBich(A)\).

\begin{thmA}\label{thmA:bicharacter-obstruction}
The cochain \(\Omega^{\operatorname{bich}}_d\) takes values in
\(\SymBich(A)\), and its cohomology class in
\(H^2(\Gamma,\SymBich(A))\) is independent of the chosen lifts \(b_x\).
The datum \(d\) lifts to a
bicharacter-normalized Hopf datum if and only if this class vanishes.
If \(A\) has odd exponent, the class vanishes for every \(d\).
\end{thmA}

We denote this class by
\[
\operatorname{ob}_{\operatorname{bich}}(d)
=
[\Omega^{\operatorname{bich}}_d].
\]
Every bicharacter lift of \(d\) is a strict lift of \(\vartheta_*d\),
but a strict lift need not arise from bicharacters. Let
\(\iota:\SymBich(A)\to B^2(A,\kk^\times)\) be the natural inclusion.
We write \(H^2(\Gamma,\iota)\) for the induced map from
\(H^2(\Gamma,\SymBich(A))\) to
\(H^2(\Gamma,B^2(A,\kk^\times))\).
Equation~\eqref{eq:bich-strict-comparison} identifies the strict
obstruction with the image of the bicharacter obstruction.
\[
H^2(\Gamma,\iota)
\bigl(\operatorname{ob}_{\operatorname{bich}}(d)\bigr)
=
\delta(\vartheta_*d).
\]
Thus the possible loss of information is measured by
\(\ker H^2(\Gamma,\iota)\). Let
\[
Q(A)=B^2(A,\kk^\times)/\SymBich(A).
\]
The coefficient sequence
\(1\to\SymBich(A)\xrightarrow{\iota}B^2(A,\kk^\times)
\to Q(A)\to1\) has connecting homomorphism
\(\partial_Q:H^1(\Gamma,Q(A))\to
H^2(\Gamma,\SymBich(A))\). Exactness gives
\[
\ker H^2(\Gamma,\iota)
=
\operatorname{im}(\partial_Q).
\]
Hence \(\operatorname{ob}_{\operatorname{bich}}(d)\) may be nonzero even
when \(\delta(\vartheta_*d)=0\).
Theorem~\ref{thm:bich-obstruction} establishes the bicharacter lifting
criterion, while Corollary~\ref{cor:odd-exponent} gives an explicit
canonical bicharacter lift when \(A\) has odd exponent.

\subsection*{Permutation modules and generalized Kac--Paljutkin algebras}

Let \(R=\Z/n\Z\), with \(n\geq2\), let \(X\) be a finite
\(\Gamma\)-set, and put \(A=R[X]\). After fixing a primitive
\(n\)-th root \(q\in\kk^\times\), bicharacters are represented by
matrices over \(R\). Proposition~\ref{prop:perm-decomposition} gives
\[
\SymBich(A)
\cong
R[X]\oplus R\!\left[\binom{X}{2}\right],
\qquad
\AltBich(A)
\cong
\bigoplus_{P\in\Gamma\backslash\binom{X}{2}}
\Ind_{\Gamma_P}^{\Gamma}R_{\varepsilon_P}.
\]
Here \(\Gamma_P\) is the setwise stabilizer of \(P\), and
\(R_{\varepsilon_P}\) is its orientation module on the two-point set
\(P\).
Consequently, Corollary~\ref{cor:perm-shapiro} computes the target
\(H^2(\Gamma,\SymBich(A))\) of the bicharacter obstruction and the
quotient \(H^1(\Gamma,\AltBich(A))\) through point and pair stabilizers.
The connecting maps associated with the orientation modules determine
the bicharacter obstruction by Remark~\ref{rem:perm-bich-obstruction};
the crossed cohomological data themselves lie in
\(Z^1(\Gamma,\AltBich(A))\).

Let \(\Opext_\Gamma(A)\) denote the group of equivalence classes of all
cocentral abelian extensions with the fixed action and trivial
coaction; unlike \(\Ecal(A,\Gamma)\), no algebra-split condition is
imposed. Put \(A^\vee=\Hom(A,\kk^\times)\). Each extension class
determines a crossed family
\(x\mapsto[\tau_x]\in H^2(A,\kk^\times)\); denote this assignment by
\(\operatorname{cl}\). Applying \(\vartheta^{-1}\) pointwise identifies
such families with elements of \(Z^1(\Gamma,\AltBich(A))\).

For permutation modules, the quadratic correction constructed in
Theorem~\ref{thm:permutation-kac-vanishing} produces a homomorphic
section and hence the split exact sequence
\[
0
\longrightarrow
H^2(\Gamma,A^\vee)
\longrightarrow
\Opext_\Gamma(A)
\xrightarrow{\vartheta_*^{-1}\circ\operatorname{cl}}
Z^1(\Gamma,\AltBich(A))
\longrightarrow0.
\]
Surjectivity of \(\vartheta_*^{-1}\circ\operatorname{cl}\) says that
every \(d\in Z^1(\Gamma,\AltBich(A))\) occurs as the cohomological datum
of an extension class in \(\Opext_\Gamma(A)\). By
Theorem~\ref{thm:general-obstruction}, \(d\) has an algebra-split
realization precisely when its strict obstruction
\(\delta(\vartheta_*d)\) vanishes.

For the natural action of \(\mathfrak S_m\) on \(A=R^m\), and for
\(i\neq j\), let \(E_{ij}\in M_m(R)\) have \(1\) in position
\((i,j)\), \(-1\) in position \((j,i)\), and zeros elsewhere. Define
\(d_{\operatorname{inv}}\) by
\[
d_{\operatorname{inv}}(s_i)=E_{i,i+1},
\qquad
d_{\operatorname{inv}}(w)
=
\sum_{\substack{a<b\\w(a)>w(b)}}E_{w(b),w(a)}.
\]
Proposition~\ref{prop:symmetric-alt-cohomology} computes
\[
H^1(\mathfrak S_m,\AltBich(R^m))
\cong
\begin{cases}
R/2R,&m=2,3,\\
R/2R\oplus R[2],&m\geq4,
\end{cases}
\]
where \(R[2]=\{r\in R\mid 2r=0\}\). The class of
\(d_{\operatorname{inv}}\) has \(R/2R\)-coordinate \(1\), and its
\(R[2]\)-coordinate is
zero when \(m\geq4\). For every \(m,n\geq2\),
Proposition~\ref{prop:inversion-strict-lift} gives a strict lift of
\(\vartheta_*d_{\operatorname{inv}}\), whereas
Proposition~\ref{prop:inversion-lift} shows that a bicharacter lift
exists precisely when \(n\) is odd.

In particular, \eqref{eq:symmetric-full-opext-explicit} gives a
noncanonical description of the full group.
\[
\Opext_{\mathfrak S_m}(R^m)
\cong
\begin{cases}
R^{\binom{m}{2}},&n\text{ odd},\\
R,&n\text{ even and }m=2,\\
R^3\oplus\Z/2\Z,&n\text{ even and }m=3,\\
R^6\oplus(\Z/2\Z)^2,&n\text{ even and }m=4,\\
R^{\binom{m}{2}}\oplus(\Z/2\Z)^3,&n\text{ even and }m\geq5.
\end{cases}
\]

Under the identification
\(\kk[\Z/n\Z]^{\otimes m}\cong\kk^{R^m}\) determined by \(q\), the
adjacent twists in the generalized Kac--Paljutkin algebra \(H_{n,m}\) constructed
in~\cite{Lomp2025} become the bicharacters
\[
c_i(a,b)=q^{a_i b_{i+1}},
\qquad
\Alt(c_i)=E_{i,i+1}.
\]
An extension class in \(\Opext_\Gamma(A)\) can be represented by a pair
\((\sigma,\tau)\), where
\(\sigma\in Z^2(\Gamma,C^1(A,\kk^\times))\) controls multiplication.
We denote the cohomology class
\([\sigma]\in H^2(\Gamma,C^1(A,\kk^\times))\) by
\(\operatorname{mult}([(\sigma,\tau)])\). By
Proposition~\ref{prop:full-kac-comparison}, this class vanishes
precisely when the extension is algebra-split.
Proposition~\ref{prop:Hnm-extension-class} identifies \([H_{n,m}]\)
inside the full extension group through the equalities
\[
\operatorname{mult}([H_{n,m}])=1,
\qquad
\operatorname{cl}([H_{n,m}])
=
\vartheta_*d_{\operatorname{inv}}.
\]
The equality \(\operatorname{mult}([H_{n,m}])=1\) records that the
multiplication component of \([H_{n,m}]\) is a coboundary.
Under the splitting of Theorem~\ref{thm:permutation-kac-vanishing}, its
class is \((0,d_{\operatorname{inv}})\). The cocycle
\(d_{\operatorname{inv}}\) has order \(n\), although its
\(H^1\)-class is trivial for odd \(n\) and has order \(2\) for even
\(n\). Thus, when \(n\) is even, \(H_{n,m}\) is algebra-split but admits
no bicharacter-normalized algebra-split representative under
equivalences inducing the identity on the kernel and the quotient.

\subsection*{Finite reductions of Coxeter representations}

Section~\ref{sec:arithmetic-coxeter} applies the obstruction theory to
finite reductions \(L_{\mathfrak a}\) of the geometric representation
of a finite-rank Coxeter system \((W,S)\) over the residue rings
\(R_{\mathfrak a}\).
Proposition~\ref{prop:coxeter-crossed} converts crossed homomorphisms
into the relations imposed on their values at the simple reflections.
Restricting to the associated linear coefficient module turns these relations
into a finite system of linear equations. A generating additive character
embeds this module into \(\AltBich(L_{\mathfrak a})\).

For finite Coxeter groups, Proposition~\ref{prop:finite-averaging}
shows that the linearized first cohomology vanishes when \(|W|\) is
invertible in the residue ring. Theorem~\ref{thm:dihedral-cohomology}
gives
\[
H^1(I_2(m),R_{\mathfrak a}^{\operatorname{sgn}})
\cong
R_{\mathfrak a}/2R_{\mathfrak a}\oplus R_{\mathfrak a}[m].
\]
for finite dihedral groups over \(R_{\mathfrak a}\), where
\(R_{\mathfrak a}[m]=\{r\in R_{\mathfrak a}\mid mr=0\}\).
For \(R=\Z/n\Z\), Proposition~\ref{prop:affine-a1} identifies the
corresponding first cohomology for the infinite dihedral group with
\(R\oplus R/2R\). For any bicharacter-normalized datum for this group,
Proposition~\ref{prop:affine-factorization} produces a finite quotient
and hence a finite-dimensional Hopf algebra.

\subsection*{Organization of the paper}
Section~\ref{sec:extensions} develops the cocentral algebra-split
extension formulas; Section~\ref{sec:obstructions} proves the obstruction
results; Section~\ref{sec:permutation} treats permutation modules, the
full extension group, and the generalized Kac--Paljutkin family; and
Section~\ref{sec:arithmetic-coxeter} treats finite reductions of Coxeter
representations.

\section{Cocentral algebra-split abelian extensions}\label{sec:extensions}

Throughout the paper, \(\kk\) is an algebraically closed field of
characteristic zero. Let \(A\) be a finite abelian group, written
additively, and let \(\Gamma\) act on \(A\) by automorphisms. All
actions are left actions.

\subsection{From abelian extensions to the algebra-split case}\label{sec:opext-background}

We recall only the part of abelian extension theory needed for our
construction.
Let
\[
\kk\longrightarrow K\xrightarrow{i} H\xrightarrow{\pi} Q
\longrightarrow \kk
\]
be an exact sequence of Hopf algebras in the usual sense of Hopf
algebra extension theory~\cite{AndruskiewitschDevoto1995,Masuoka2002Extensions}.
In the cleft-cocleft setting,
after choosing cleaving maps, the extension is described on the vector
space \(K\otimes Q\) by four pieces of data
\((\cdot,\sigma,\rho,\tau)\).
The symbols \(\cdot\) and \(\rho\) denote a weak action and a weak
coaction, while \(\sigma\) and \(\tau\) are respectively an algebra
cocycle and a coalgebra cocycle. The corresponding Hopf algebra is a bicrossed
product \(K^\tau\#_\sigma Q\); the cocycle \(\sigma\) twists the
multiplication and \(\tau\) twists the
comultiplication~\cite[Theorem~2.20 and Theorem~3.2.14]{AndruskiewitschDevoto1995}.

When \(K\) is commutative and \(Q\) is cocommutative one obtains the
abelian Hopf algebra extensions. If the weak action \(\cdot\) and weak
coaction \(\rho\) are fixed, the remaining extension classes form the
abelian group \(\Opext(Q,K;\cdot,\rho)\)~\cite{Schauenburg2002Kac}.
The general extension theory serves only as background.

For cocentral algebra-split extensions, our arguments use only the
explicit cocycles determined by the fixed action and coaction. We take
\(K=\kk^A\) and \(Q=\kk\Gamma\), where \(A\) is a finite abelian group
and \(\Gamma\) acts on \(A\). The action gives the action part.
Cocentrality forces the coaction part to be trivial,
\(\rho_{\mathrm{triv}}(x)=x\otimes 1\) for \(x\in\Gamma\). For finite
\(\Gamma\), this agrees with the group-algebra criterion in
\cite[Lemma~3.3]{Natale2010TY}. Algebra-splitness means that the algebra
cocycle can be taken to be trivial, \(\sigma=1\).
Thus the only remaining extension datum is the coalgebra cocycle
\(\tau\).

Since \(\kk^A\otimes\kk^A\cong \kk^{A\times A}\), the cocycle \(\tau\)
is a family of functions \(\tau_x\) from
\(A\times A\) to \(\kk^\times\), indexed by \(x\in\Gamma\).
The compatibility identities for \((\sigma,\tau)\), after imposing
\(\sigma=1\) and trivial \(\rho\), reduce exactly to the conditions that
each \(\tau_x\) is a normalized group \(2\)-cocycle on \(A\) and that
\[
\tau_{xy}=\tau_x(x\cdot\tau_y).
\]
Thus the extensions considered in this paper form the cocentral
algebra-split subfamily of the usual abelian-extension group
\(\Opext(\kk\Gamma,\kk^A;\cdot,\rho_{\mathrm{triv}})\), namely the classes
represented by data with trivial coaction and trivial
algebra cocycle.
Proposition~\ref{prop:tau-conditions} constructs the corresponding Hopf
algebras from the functions \(\tau_x\),
Proposition~\ref{prop:equivalence-sections} describes the effect of
changing the section, and
Theorem~\ref{thm:general-obstruction} gives the obstruction to lifting
a crossed homomorphism \(\Gamma\to H^2(A,\kk^\times)\) to such strict
cocycle data.

\subsection{Strict Hopf data and Hopf algebra formulas}

Proposition~\ref{prop:tau-conditions} constructs these Hopf algebras
explicitly as cocentral algebra-split abelian extensions, which appear
naturally in the theory of abelian extensions and matched pairs;
see~\cite{Kac1968,Takeuchi1981,Montgomery1993}.

Let \(\kk^A\) be the function algebra on \(A\). We write
\(\{e_\lambda\}_{\lambda\in A}\) for the standard basis of orthogonal idempotents.
\[
e_\lambda e_\mu=\delta_{\lambda,\mu}e_\lambda,
\qquad
1_{\kk^A}=\sum_{\lambda\in A}e_\lambda.
\]
The Hopf algebra structure on \(\kk^A\) is
\[
\Delta(e_\lambda)=\sum_{\alpha+\beta=\lambda}e_\alpha\otimes e_\beta,
\qquad
\varepsilon(e_\lambda)=\delta_{\lambda,0}.
\]

\begin{definition}
A Hopf algebra map \(\pi\) from \(H\) to \(Q\) is called
\emph{cocentral} if
\[
\pi(h_{(1)})\otimes h_{(2)}
=
\pi(h_{(2)})\otimes h_{(1)}
\]
for all \(h\in H\). An extension
\[
\kk\longrightarrow \kk^A\longrightarrow H
\xrightarrow{\pi}
\kk\Gamma\longrightarrow \kk
\]
is called a \emph{cocentral algebra-split abelian extension} if \(\pi\) is
cocentral, \(\kk^A=H^{\operatorname{co}\kk\Gamma}\), and the extension
is split in the algebraic sense. In the present setting, this means that
\(\pi\) admits a unital algebra map \(\gamma:\kk\Gamma\longrightarrow H\)
which is a right \(\kk\Gamma\)-comodule section. Equivalently, after
choosing such a section, the algebra cocycle is trivial and
\(H\simeq \kk^A\rtimes\kk\Gamma\) as an algebra.
\end{definition}

Let \(\gamma\) be such a section from \(\kk\Gamma\) to \(H\), and write
\(u_x=\gamma(x)\) for \(x\in\Gamma\).
Then
\[
u_1=1,
\qquad
u_x u_y=u_{xy},
\qquad
\pi(u_x)=x,
\]
and \((\id\otimes\pi)\Delta(u_x)=u_x\otimes x\).
We assume that the section implements the given action of \(\Gamma\) on \(A\), so that
\(u_x e_\lambda u_x^{-1}=e_{x\cdot\lambda}\). Thus, as an algebra,
\(H\simeq \kk^A\rtimes \kk\Gamma\),
with multiplication
\begin{equation}\label{eq:smash-product}
(e_\lambda u_x)(e_\mu u_y)
=
\delta_{\lambda,x\cdot\mu}\,e_\lambda u_{xy}.
\end{equation}

Since the algebra structure is already determined by the action of
\(\Gamma\) on \(A\), the remaining Hopf-theoretic information is
encoded in the coproduct of the elements \(u_x\).

For \(x\in\Gamma\), put
\[
H_x
=
\{h\in H\mid(\id\otimes\pi)\Delta(h)=h\otimes x\}.
\]
The smash-product decomposition gives
\(H=\bigoplus_{x\in\Gamma}H_x\) and \(H_x=\kk^A u_x\).
Coassociativity gives \(\Delta(H_x)\subseteq H\otimes H_x\).
Cocentrality identifies the corresponding left degree with \(x\), and
coassociativity then gives \(\Delta(H_x)\subseteq H_x\otimes H\). Therefore
\[
\Delta(H_x)
\subseteq
(H\otimes H_x)\cap(H_x\otimes H)
=
H_x\otimes H_x.
\]
Hence there is a unique
\(J_x\in\kk^A\otimes\kk^A\) such that
\[
\Delta(u_x)=J_x(u_x\otimes u_x).
\]
Since \(\Delta(u_x)\) and \(u_x\otimes u_x\) are invertible, a standard
finite-dimensional argument gives
\(J_x\in(\kk^A\otimes\kk^A)^\times\).
Write
\[
J_x
=
\sum_{\alpha,\beta\in A}
\tau_x(\alpha,\beta)\,
e_\alpha\otimes e_\beta.
\]
Then
\begin{equation}\label{eq:coproduct-tau}
\Delta(e_\lambda u_x)
=
\sum_{\alpha+\beta=\lambda}
\tau_x(\alpha,\beta)
(e_\alpha u_x)\otimes(e_\beta u_x).
\end{equation}

\begin{definition}\label{def:strict-hopf-datum}
We write \(Z^2(A,\kk^\times)\) for the group of normalized group
\(2\)-cocycles on \(A\). Let \(\Gamma\) act on
\(Z^2(A,\kk^\times)\) by
\[
(x\cdot\omega)(\alpha,\beta)
=
\omega(x^{-1}\cdot\alpha,x^{-1}\cdot\beta).
\]
A \emph{strict Hopf datum} is a crossed homomorphism
\(\tau\in Z^1(\Gamma,Z^2(A,\kk^\times))\).
Explicitly, it is a family of normalized \(2\)-cocycles
\(\tau_x\) on \(A\) satisfying
\begin{equation}\label{eq:crossed-tau}
\tau_{xy}(\alpha,\beta)
=
\tau_x(\alpha,\beta)
\tau_y(x^{-1}\cdot\alpha,x^{-1}\cdot\beta).
\end{equation}
\end{definition}

\begin{proposition}\label{prop:tau-conditions}
The formulas~\eqref{eq:smash-product} and~\eqref{eq:coproduct-tau}
define a cocentral algebra-split abelian Hopf algebra extension if and only if
\(\tau\) is a strict Hopf datum. In that case
\(\varepsilon(e_\lambda u_x)=\delta_{\lambda,0}\), and
\begin{equation}\label{eq:antipode}
S(e_\lambda u_x)
=
\tau_x(\lambda,-\lambda)^{-1}
e_{-x^{-1}\cdot\lambda}u_{x^{-1}}.
\end{equation}
\end{proposition}

\begin{proof}
Assume first that formulas~\eqref{eq:smash-product}
and~\eqref{eq:coproduct-tau} define a Hopf algebra. Applying
coassociativity to \(u_x\) and comparing the coefficient of
\((e_\alpha u_x)\otimes(e_\beta u_x)\otimes(e_\gamma u_x)\) gives
\[
\tau_x(\alpha,\beta)\tau_x(\alpha+\beta,\gamma)
=
\tau_x(\beta,\gamma)\tau_x(\alpha,\beta+\gamma).
\]
The counit identities give
\(\tau_x(0,\lambda)=\tau_x(\lambda,0)=1\).
Thus each \(\tau_x\) is a normalized \(2\)-cocycle on \(A\).

Since \(u_xu_y=u_{xy}\), multiplicativity of the coproduct, followed by
comparison of the coefficient of
\((e_\alpha u_{xy})\otimes(e_\beta u_{xy})\), gives
\(\tau_{xy}(\alpha,\beta)=\tau_x(\alpha,\beta)
\tau_y(x^{-1}\!\cdot\alpha,x^{-1}\!\cdot\beta)\). Thus \(\tau\) is a
strict Hopf datum.

Conversely, assume that \(\tau\) is a strict Hopf datum. The algebra
\(\kk^A\rtimes\kk\Gamma\) is associative by construction, normalization
gives the counit identities, and the \(2\)-cocycle identity gives
coassociativity. Also, \eqref{eq:crossed-tau} gives \(\tau_1=1\) and,
together with the smash-product multiplication, implies that \(\Delta\)
is multiplicative on the basis elements \(e_\lambda u_x\). Indeed, if
\(\lambda=x\cdot\mu\), the coefficient of
\((e_\alpha u_{xy})\otimes(e_\beta u_{xy})\) in
\(\Delta(e_\lambda u_x)\Delta(e_\mu u_y)\) is
\(\tau_x(\alpha,\beta)
\tau_y(x^{-1}\!\cdot\alpha,x^{-1}\!\cdot\beta)\), which equals
\(\tau_{xy}(\alpha,\beta)\) by~\eqref{eq:crossed-tau}; if
\(\lambda\ne x\cdot\mu\), both sides vanish.

Since \(\tau_1=1\), the copy of \(\kk^A\) is a Hopf subalgebra. The
smash-product formula makes \(\pi\) an algebra map, while normalization
gives \((\pi\otimes\pi)\Delta=\Delta\pi\) and
\(\varepsilon=\varepsilon_\Gamma\pi\), where \(\varepsilon_\Gamma\)
is the counit of \(\kk\Gamma\).
For the quotient map \(\pi(e_\lambda u_x)=\delta_{\lambda,0}x\),
normalization gives
\((\pi\otimes\id)\Delta(e_\lambda u_x)
=(\pi\otimes\id)\Delta^{\mathrm{op}}(e_\lambda u_x)
=x\otimes e_\lambda u_x\), so \(\pi\) is cocentral. Its induced coaction is
\((\id\otimes\pi)\Delta(e_\lambda u_x)=e_\lambda u_x\otimes x\), so the
coinvariants are \(\kk^A\); the smash-product basis also gives
\(\ker\pi=H\,(\kk^A)^+\), where
\((\kk^A)^+=\ker(\varepsilon|_{\kk^A})\). Thus the sequence is exact. Finally,
\(\gamma:\kk\Gamma\to H\), \(\gamma(x)=u_x\), is a unital algebra map
and a right \(\kk\Gamma\)-comodule section of \(\pi\), so the extension is
algebra-split.

With \(S\) as in~\eqref{eq:antipode}, one has
\(m(S\otimes\id)\Delta(e_\lambda u_x)=\delta_{\lambda,0}1\).
Indeed, the only nonzero summands have \(\beta=-\alpha\). The same
calculation for \(m(\id\otimes S)\Delta\) uses
\(\tau_x(\alpha,-\alpha)=\tau_x(-\alpha,\alpha)\), which follows from
the cocycle identity with \((\alpha,\beta,\gamma)=(\alpha,-\alpha,\alpha)\).
Thus \(S\) is the antipode, and the formulas define the required Hopf
algebra.

\end{proof}

Thus strict Hopf data parametrize Hopf structures on the fixed smash
product algebra \(\kk^A\rtimes\kk\Gamma\). We denote the Hopf algebra
defined by \(\Gamma\), \(A\), and \(\tau\) by \(H(\Gamma,A,\tau)\);
when the strict datum is a crossed family of bicharacters denoted by
\(b\), we write \(H(\Gamma,A,b)\).
Different choices of the lifts \(u_x\) may lead to different strict
data defining equivalent extensions. This leads to the obstruction
theory developed in Section~\ref{sec:obstructions}.

\section{Obstructions and bicharacter lifts}\label{sec:obstructions}

For a fixed extension, a homogeneous section \(u_x=\gamma(x)\) determines
the cocycles \(\tau_x\). Choosing representatives of prescribed
cohomology classes is a separate part of the lifting problem. We treat
these two choices separately.
Subsection~\ref{sec:equivalence-obstruction} describes equivalence of
strict data and the general lifting obstruction, while
Subsection~\ref{sec:bicharacter} studies compatible
bicharacter-normalized representatives.

\subsection{Equivalence of strict data and the general obstruction}\label{sec:equivalence-obstruction}

We begin with the dependence on the chosen lifts
\(u_x=\gamma(x)\), including the resulting equivalence relation and the
obstruction to lifting cohomological data.

We fix \(A\), \(\Gamma\), and the action of \(\Gamma\) on \(A\).
As in Definition~\ref{def:strict-hopf-datum},
\(Z^2(A,\kk^\times)\) denotes normalized group \(2\)-cocycles. We also
write \(C^1(A,\kk^\times)\) for the group of normalized functions
on \(A\) with values in \(\kk^\times\). For
\(f\in C^1(A,\kk^\times)\), define
\[
(\partial f)(\alpha,\beta)
=
\frac{f(\alpha)f(\beta)}
     {f(\alpha+\beta)}.
\]
The coboundaries form
\(B^2(A,\kk^\times)=\partial C^1(A,\kk^\times)\), and we denote the
quotient \(Z^2(A,\kk^\times)/B^2(A,\kk^\times)\) by
\(H^2(A,\kk^\times)\).
The action of \(\Gamma\) on \(A\) induces actions on these groups; for
instance \((x\cdot f)(\lambda)=f(x^{-1}\cdot\lambda)\) for
\(f\in C^1(A,\kk^\times)\).
We reserve \(\partial\) for cochains of \(A\). The group-cohomology
differential on cochains of \(\Gamma\) is denoted by
\(\partial_\Gamma\), and is written multiplicatively when the
coefficient group is written multiplicatively.

\begin{definition}\label{def:equivalent-strict-data}
Two strict Hopf data
\(\tau,\tau'\in Z^1(\Gamma,Z^2(A,\kk^\times))\)
are called \emph{equivalent} if there exists
\(f\in Z^1(\Gamma,C^1(A,\kk^\times))\) such that
\(\tau'_x=\tau_x\,\partial f_x\) for every \(x\in\Gamma\).
\end{definition}

\begin{proposition}\label{prop:equivalence-sections}
The equivalence relation of Definition~\ref{def:equivalent-strict-data}
coincides with equivalence by Hopf algebra isomorphisms inducing the
identity on the kernel \(\kk^A\) and the quotient \(\kk\Gamma\).
\end{proposition}

\begin{proof}
Let \(H_\tau\) and \(H_{\tau'}\) denote the two Hopf algebras. Assume
\(\tau'_x=\tau_x\partial f_x\) as in
Definition~\ref{def:equivalent-strict-data}. Then
\[
\Phi(e_\lambda u_x)
=
f_x(\lambda)e_\lambda u'_x
\]
is multiplicative by the crossed-homomorphism identity for \(f\). For
\(\alpha+\beta=\lambda\), the equality
\(f_x(\lambda)\tau'_x(\alpha,\beta)
=\tau_x(\alpha,\beta)f_x(\alpha)f_x(\beta)\) shows that \(\Phi\)
intertwines the coproducts. Normalization gives \(f_x(0)=1\), and the
crossed identity gives \(f_1=1\); hence \(\Phi\) fixes the quotient and
the kernel. Replacing \(f_x\) with \(f_x^{-1}\) gives its inverse.

Conversely, let \(\Phi\) be such an isomorphism from \(H_\tau\) to
\(H_{\tau'}\). Since \(\Phi\) is the identity on \(\kk^A\) and
on the quotient, it preserves the homogeneous component over \(x\), and
also the one-dimensional left \(\kk^A\)-weight spaces
\(e_\lambda(H_\tau)_x=\kk e_\lambda u_x\). Hence
\[
\Phi(e_\lambda u_x)=f_x(\lambda)e_\lambda u'_x
\]
for uniquely determined scalars \(f_x(\lambda)\in\kk^\times\). The
identity on \(\kk^A\) gives \(f_1(\lambda)=1\), and the quotient
condition gives \(f_x(0)=1\). Compatibility with multiplication gives
\(f_{xy}(\lambda)=f_x(\lambda)f_y(x^{-1}\cdot\lambda)\),
so \(f\in Z^1(\Gamma,C^1(A,\kk^\times))\), while compatibility with the
coproduct gives \(\tau'_x=\tau_x\partial f_x\).
Equivalently, replacing the section \(x\mapsto u_x\) by
\[
u_x
\longmapsto
\sum_{\lambda\in A}
f_x(\lambda)^{-1}e_\lambda u_x
\]
produces precisely this transformation of the strict datum.
\end{proof}

Thus the group of equivalence classes with fixed \(A\), \(\Gamma\), and
action is
\begin{equation}\label{eq:E}
\Ecal(A,\Gamma)
=
\frac{Z^1(\Gamma,Z^2(A,\kk^\times))}
     {\partial Z^1(\Gamma,C^1(A,\kk^\times))}.
\end{equation}

\begin{remark}
For the fixed action of \(\Gamma\) on \(A\) and the trivial coaction,
\(\Ecal(A,\Gamma)\) is the subgroup of
\(\Opext(\kk\Gamma,\kk^A;\cdot,\rho_{\mathrm{triv}})\) consisting of
classes represented by the trivial algebra cocycle \(\sigma=1\).
\end{remark}

The exact sequence
\[
1
\longrightarrow
B^2(A,\kk^\times)
\longrightarrow
Z^2(A,\kk^\times)
\longrightarrow
H^2(A,\kk^\times)
\longrightarrow
1
\]
induces a map
\(\Ecal(A,\Gamma)\longrightarrow Z^1(\Gamma,H^2(A,\kk^\times))\).
A strict datum \(\tau\) maps to \(d_\tau(x)=[\tau_x]\).

Given \(d\in Z^1(\Gamma,H^2(A,\kk^\times))\), choose representatives
\(\widetilde\tau_x\in Z^2(A,\kk^\times)\) such that
\([\widetilde\tau_x]=d(x)\) and \(\widetilde\tau_1=1\), and define
\[
\Omega_d(x,y)
=
\widetilde\tau_x
(x\cdot\widetilde\tau_y)
\widetilde\tau_{xy}^{-1}.
\]
Since \(d\) is a crossed homomorphism,
\(\Omega_d(x,y)\in B^2(A,\kk^\times)\).

Theorem~\ref{thm:general-obstruction} identifies the obstruction to
lifting a cohomological datum to a strict Hopf datum. This obstruction
is the connecting class for the coefficient sequence
\(1\to B^2(A,\kk^\times)\to Z^2(A,\kk^\times)
\to H^2(A,\kk^\times)\to1\); compare with the construction of
connecting maps in group cohomology~\cite[Chapter~IV]{Brown1982}.

\begin{theorem}\label{thm:general-obstruction}
The assignment
\[
\delta(d)
=
[\Omega_d]
\in
H^2(\Gamma,B^2(A,\kk^\times))
\]
is a well-defined homomorphism. Moreover, \(d\) lifts to a strict Hopf datum if and
only if \(\delta(d)\) is trivial.
\end{theorem}

\begin{proof}
Since \(d\) is crossed, the defining formula shows that \(\Omega_d(x,y)\) lies in
\(B^2(A,\kk^\times)\). The standard defect-of-lifts calculation gives
\(\partial_\Gamma\Omega_d=1\), so
\(\Omega_d\in Z^2(\Gamma,B^2(A,\kk^\times))\).

If the representatives are replaced by
\(\widetilde\tau_x'=\widetilde\tau_x b_x\), with
\(b_x\in B^2(A,\kk^\times)\), then the new cochain is
\(\Omega'_d=\Omega_d\,\partial_\Gamma b\), where
\((\partial_\Gamma b)(x,y)=b_x(x\cdot b_y)b_{xy}^{-1}\).
Hence \(\delta(d)\) is independent of the chosen representatives.

If \(d,d'\in Z^1(\Gamma,H^2(A,\kk^\times))\) are represented by
\(\widetilde\tau\) and \(\widetilde\sigma\), respectively, then
\(\widetilde\tau_x\widetilde\sigma_x\) represents \((dd')_x\), and
commutativity of the coefficient groups gives
\(\Omega_{dd'}=\Omega_d\Omega_{d'}\).
Thus \(\delta(dd')=\delta(d)\delta(d')\), so \(\delta\) is a
homomorphism.

If \(\delta(d)\) is trivial, there is a \(1\)-cochain \(b\) with values in
\(B^2(A,\kk^\times)\) such that \(\Omega_d\,\partial_\Gamma b=1\). Replacing
\(\widetilde\tau_x\) by \(\widetilde\tau_x b_x\), we may assume that
\(\Omega_d(x,y)=1\). The family \(\{\widetilde\tau_x\}\) then satisfies
\(\widetilde\tau_{xy}=\widetilde\tau_x(x\cdot\widetilde\tau_y)\), and
therefore defines a strict Hopf datum
\(\tau\in Z^1(\Gamma,Z^2(A,\kk^\times))\) lifting \(d\).

Conversely, if \(d\) admits a lift
\(\tau\in Z^1(\Gamma,Z^2(A,\kk^\times))\), then choosing
\(\widetilde\tau_x=\tau_x\) gives \(\Omega_d(x,y)=1\),
so \(\delta(d)\) is trivial.
\end{proof}

As the connecting homomorphism in
Theorem~\ref{thm:general-obstruction}, \(\delta\) is unchanged when
\(d\) is multiplied by a principal crossed homomorphism. Hence it
factors through
\[
H^1(\Gamma,H^2(A,\kk^\times))
\longrightarrow
H^2(\Gamma,B^2(A,\kk^\times)).
\]
We nevertheless retain \(Z^1\) in the lifting sequence because the
classification with fixed kernel, quotient, and action records the
actual crossed homomorphism \(d\), not only its class modulo principal
crossed homomorphisms.

\begin{corollary}\label{cor:general-obstruction-exact}
There is an exact sequence
\[
1
\longrightarrow
\frac{Z^1(\Gamma,B^2(A,\kk^\times))}
     {\partial Z^1(\Gamma,C^1(A,\kk^\times))}
\longrightarrow
\Ecal(A,\Gamma)
\longrightarrow
Z^1(\Gamma,H^2(A,\kk^\times))
\xrightarrow{\delta}
H^2(\Gamma,B^2(A,\kk^\times)).
\]
\end{corollary}

\begin{proof}
The inclusion \(B^2(A,\kk^\times)\hookrightarrow
Z^2(A,\kk^\times)\) induces a map between the two quotients. The
induced map is injective because
\(\partial Z^1(\Gamma,C^1(A,\kk^\times))\) is contained in
\(Z^1(\Gamma,B^2(A,\kk^\times))\), so both quotients use the same
subgroup. Exactness at
\(\Ecal(A,\Gamma)\) follows because the kernel of the projection from
\(Z^1(\Gamma,Z^2(A,\kk^\times))\) to
\(Z^1(\Gamma,H^2(A,\kk^\times))\) is
\(Z^1(\Gamma,B^2(A,\kk^\times))\). Exactness at
\(Z^1(\Gamma,H^2(A,\kk^\times))\) is precisely the lifting criterion
of Theorem~\ref{thm:general-obstruction}.
\end{proof}

\subsection{Bicharacter lifts and the bicharacter obstruction}\label{sec:bicharacter}

The obstruction in Theorem~\ref{thm:general-obstruction} governs strict
lifts. We also study the bicharacter-normalized lifting problem, in which
the representatives are bilinear. Any such lift is a strict lift.

Let \(\Bich(A)\) be the group of bicharacters
\(A\times A\to\kk^\times\). Define
\[
\Alt(b)(\alpha,\beta)
=
\frac{b(\alpha,\beta)}
     {b(\beta,\alpha)}.
\]
Let \(\SymBich(A)\) be the subgroup of symmetric bicharacters and
\(\AltBich(A)\) the group of alternating bicharacters.
Here alternating means that \(a(\lambda,\lambda)=1\) for every
\(\lambda\in A\), a condition that is relevant in the presence of
\(2\)-torsion.

For finite abelian \(A\) over an algebraically closed field \(\kk\) of
characteristic zero, the inclusion of bicharacters and the commutator
map induce natural isomorphisms~\cite[Proposition~2.6]{Tambara2000}
\[
\Bich(A)/\SymBich(A)
\cong
H^2(A,\kk^\times)
\cong
\AltBich(A).
\]
Thus each cohomology class has a bicharacter representative, and its
class is recovered from the alternating part. The corresponding short
exact sequence of \(\Gamma\)-modules is
\[
1
\longrightarrow
\SymBich(A)
\longrightarrow
\Bich(A)
\xrightarrow{\Alt}
\AltBich(A)
\longrightarrow
1.
\]

\begin{definition}
A \emph{bicharacter-normalized Hopf datum} is a crossed homomorphism
\(b\in Z^1(\Gamma,\Bich(A))\).
It defines a strict Hopf datum by taking \(\tau_x=b_x\). If
\(d\in Z^1(\Gamma,\AltBich(A))\),
a \emph{bicharacter lift} of \(d\) is a bicharacter-normalized Hopf
datum \(b\) such that \(\Alt(b_x)=d_x\) for every \(x\in\Gamma\).
\end{definition}

Let \(d\in Z^1(\Gamma,\AltBich(A))\). Choose lifts
\(b_x\in\Bich(A)\) satisfying \(\Alt(b_x)=d_x\) and \(b_1=1\). Their defect
\[
\Omega^{\operatorname{bich}}_d(x,y)
=
b_x(x\cdot b_y)b_{xy}^{-1}
\]
has trivial alternating part and therefore lies in \(\SymBich(A)\).

Theorem~\ref{thm:bich-obstruction} shows that the expression
\[
\operatorname{ob}_{\operatorname{bich}}(d)
=
[\Omega^{\operatorname{bich}}_d]
\in
H^2(\Gamma,\SymBich(A))
\]
is the connecting obstruction for the coefficient sequence
\(1\to\SymBich(A)\to\Bich(A)
\xrightarrow{\Alt}\AltBich(A)\to1\).

\begin{theorem}\label{thm:bich-obstruction}
The cochain \(\Omega^{\operatorname{bich}}_d\) is a \(2\)-cocycle with
values in \(\SymBich(A)\). The assignment
\(d\longmapsto\operatorname{ob}_{\operatorname{bich}}(d)\)
is a well-defined homomorphism. The obstruction
\(\operatorname{ob}_{\operatorname{bich}}(d)\) vanishes if and only if
\(d\) admits a bicharacter lift. If the obstruction vanishes, the set
of such lifts is a torsor under \(Z^1(\Gamma,\SymBich(A))\).
\end{theorem}

\begin{proof}
Since \(d\) is a crossed homomorphism and \(\Alt\) is
\(\Gamma\)-equivariant, we have
\[
\Alt\bigl(\Omega^{\operatorname{bich}}_d(x,y)\bigr)
=
d_x(x\cdot d_y)d_{xy}^{-1}=1.
\]
Thus \(\Omega^{\operatorname{bich}}_d(x,y)\in\SymBich(A)\). The
standard defect-of-lifts calculation gives
\(\partial_\Gamma\Omega^{\operatorname{bich}}_d=1\). Hence
\(\Omega^{\operatorname{bich}}_d\) is a \(2\)-cocycle with values in
\(\SymBich(A)\).

If the lifts are replaced by \(b'_x=b_x s_x\), with
\(s_x\in\SymBich(A)\), then
\(\Omega'^{\operatorname{bich}}_d
=\Omega^{\operatorname{bich}}_d\,\partial_\Gamma s\).
Thus the cohomology class is independent of the chosen lifts.

If \(d,d'\) are represented by bicharacter families \(b\) and \(c\),
respectively, then \(bc\) represents \(dd'\), and
\(\Omega^{\operatorname{bich}}_{dd'}
=\Omega^{\operatorname{bich}}_d\Omega^{\operatorname{bich}}_{d'}\).
Hence the obstruction assignment is a homomorphism.

The class vanishes precisely when a \(1\)-cochain
\(s\in C^1(\Gamma,\SymBich(A))\) satisfies
\(\Omega^{\operatorname{bich}}_d\,\partial_\Gamma s=1\). Replacing
\(b_x\) by \(b_x s_x\) then gives \(b_{xy}=b_x(x\cdot b_y)\),
which is exactly the crossed relation for a bicharacter-normalized Hopf
datum. The converse is obtained by taking the lifts to be the crossed
homomorphism itself.

If \(b\) and \(b'\) are two crossed lifts of \(d\), then
\(b'_x b_x^{-1}\in\SymBich(A)\) and the crossed identities imply that
\((b'_x b_x^{-1})_x\) is a \(1\)-cocycle. Conversely, multiplication by
any element of \(Z^1(\Gamma,\SymBich(A))\) sends one crossed lift of
\(d\) to another. This is the torsor statement.
\end{proof}

Thus the bicharacter obstruction induces
\[
H^1(\Gamma,\AltBich(A))
\longrightarrow
H^2(\Gamma,\SymBich(A)).
\]
The Hopf data themselves remain actual crossed homomorphisms.

The identifications in~\cite[Proposition~2.6]{Tambara2000} imply that
every symmetric bicharacter is a coboundary. Let
\[
\iota:\SymBich(A)\longrightarrow B^2(A,\kk^\times)
\]
be the resulting inclusion, and let
\[
\vartheta:\AltBich(A)\xrightarrow{\sim}H^2(A,\kk^\times)
\]
be the \(\Gamma\)-equivariant isomorphism characterized by
\(\vartheta(\Alt(b))=[b]\).

For every
\(d\in Z^1(\Gamma,\AltBich(A))\), one has
\begin{equation}\label{eq:bich-strict-comparison}
H^2(\Gamma,\iota)
\bigl(\operatorname{ob}_{\operatorname{bich}}(d)\bigr)
=
\delta(\vartheta_*d),
\qquad
\vartheta_*d=\vartheta\circ d.
\end{equation}
Indeed, if \(b_x\) are bicharacter representatives satisfying
\(\Alt(b_x)=d_x\), then they are also normalized \(2\)-cocycle representatives of
\(\vartheta(d_x)\), and both sides are represented by the same defect
\(b_x(x\cdot b_y)b_{xy}^{-1}\),
viewed respectively in \(\SymBich(A)\) and in
\(B^2(A,\kk^\times)\). Consequently, a nontrivial bicharacter obstruction
rules out a strict lift precisely when its image under
\(H^2(\Gamma,\iota)\) is nontrivial.

\begin{definition}
Let
\[
\Quad(A)
=
\{
q\in C^1(A,\kk^\times)
\mid
\partial q\in\SymBich(A)
\}.
\]
We call its elements quadratic cochains.
\end{definition}

The possible discrepancy between the two obstructions is measured by
\[
Q(A)
=
B^2(A,\kk^\times)/\SymBich(A).
\]
Indeed, the exact sequence
\[
1
\longrightarrow
\SymBich(A)
\xrightarrow{\iota}
B^2(A,\kk^\times)
\longrightarrow
Q(A)
\longrightarrow
1
\]
gives
\[
\ker H^2(\Gamma,\iota)
=
\operatorname{im}\!\left(
H^1(\Gamma,Q(A))
\longrightarrow
H^2(\Gamma,\SymBich(A))
\right).
\]
Moreover, since \(\partial:C^1(A,\kk^\times)\to
B^2(A,\kk^\times)\) is surjective and the inverse image of
\(\SymBich(A)\) is \(\Quad(A)\), it induces a canonical isomorphism
\[
Q(A)
\cong
C^1(A,\kk^\times)/\Quad(A).
\]

Write \(A^\vee=\Hom(A,\kk^\times)\). There is an exact sequence
\[
1
\longrightarrow
A^\vee
\longrightarrow
\Quad(A)
\xrightarrow{\partial}
\SymBich(A)
\longrightarrow
1.
\]
Indeed, the kernel of \(\partial\) is precisely the group of
characters of \(A\). Surjectivity follows because a symmetric
bicharacter has trivial commutator, hence represents the trivial class
in \(H^2(A,\kk^\times)\)~\cite[Proposition~2.6]{Tambara2000}, and is
therefore a coboundary of a normalized \(1\)-cochain.

\begin{proposition}\label{prop:bich-equivalence}
Let \(b\) and \(b'\) be bicharacter lifts of the same datum \(d\).
They define equivalent Hopf extensions if and only if there exists
\(q\in Z^1(\Gamma,\Quad(A))\) such that
\(b'_x=b_x\partial q_x\)
for every \(x\in\Gamma\).
\end{proposition}

\begin{proof}
Set \(s_x=b'_x b_x^{-1}\). Since \(b\) and \(b'\) lift the same datum
and are crossed homomorphisms, \(s\in Z^1(\Gamma,\SymBich(A))\).
Proposition~\ref{prop:equivalence-sections} says that they are equivalent
precisely when \(s_x=\partial q_x\) for some
\(q\in Z^1(\Gamma,C^1(A,\kk^\times))\). Since every \(s_x\) is
symmetric, any such \(q\) takes values in \(\Quad(A)\). Conversely, if
\(q\in Z^1(\Gamma,\Quad(A))\) and \(b'_x=b_x\partial q_x\), then \(q\)
also belongs to \(Z^1(\Gamma,C^1(A,\kk^\times))\), so
Proposition~\ref{prop:equivalence-sections} gives the equivalence.
\end{proof}

\begin{remark}
If the set of bicharacter lifts of \(d\) is nonempty, equivalently if
\(\operatorname{ob}_{\operatorname{bich}}(d)\) is trivial, then it is a torsor
under \(Z^1(\Gamma,\SymBich(A))\).
By Proposition~\ref{prop:bich-equivalence}, equivalence quotients by
the image of \(\partial Z^1(\Gamma,\Quad(A))\). Thus bicharacter lifts
modulo equivalence are controlled by
\[
\frac{Z^1(\Gamma,\SymBich(A))}
     {\partial Z^1(\Gamma,\Quad(A))}.
\]
\end{remark}

\subsection{The odd exponent case and canonical lifts}\label{sec:odd}

The obstruction in Theorem~\ref{thm:bich-obstruction}
vanishes automatically when \(A\) has odd exponent.

Assume that \(A\) has odd exponent \(n\), and let
\(h=2^{-1}\in\Z/n\Z\). Every value of an alternating bicharacter on
\(A\) is an \(n\)-th root of unity, so for \(a\in\AltBich(A)\) the
pointwise power
\[
s(a)=a^h
\]
is well defined and is again a bicharacter. Since
\(a(\mu,\lambda)=a(\lambda,\mu)^{-1}\), one has
\(\Alt(s(a))=a^{2h}=a\). Thus
\(s:\AltBich(A)\to\Bich(A)\) is a group homomorphism and a
\(\Gamma\)-equivariant section of \(\Alt\).

\begin{corollary}\label{cor:odd-exponent}
If \(A\) has odd exponent \(n\), every
\(d\in Z^1(\Gamma,\AltBich(A))\)
admits the canonical half-alternating lift
\[
\widehat d_x=d_x^h,
\qquad
h=2^{-1}\in\Z/n\Z.
\]
In particular, both
\(\operatorname{ob}_{\operatorname{bich}}(d)\) and
\(\delta(\vartheta_*d)\) are trivial.
\end{corollary}

\begin{proof}
Applying the \(\Gamma\)-equivariant homomorphism \(s(a)=a^h\)
pointwise gives
\(\widehat d_{xy}=s(d_x(x\cdot d_y))
=\widehat d_x(x\cdot\widehat d_y)\). Thus \(\widehat d\) is a
bicharacter lift of \(d\), so
Theorem~\ref{thm:bich-obstruction} shows that
\(\operatorname{ob}_{\operatorname{bich}}(d)\) is trivial. It is also a
strict lift, and Theorem~\ref{thm:general-obstruction} shows that
\(\delta(\vartheta_*d)\) is trivial.
\end{proof}

\begin{remark}
For the classification of extensions, the resulting lifts must still be
quotiented by the equivalence in Proposition~\ref{prop:bich-equivalence}.
\end{remark}

\section{Permutation modules}\label{sec:permutation}

We apply the bicharacter obstruction of
Theorem~\ref{thm:bich-obstruction} to permutation modules. In this case
the coefficient modules \(\SymBich(A)\) and \(\AltBich(A)\) can be
described explicitly, and the resulting decomposition reduces the
relevant cohomology groups to stabilizer cohomology. We specialize
to the symmetric group and the crossed homomorphism
\(d_{\operatorname{inv}}\).

The permutation-module cohomological data remain crossed homomorphisms;
Shapiro's lemma computes only their classes modulo principal crossed
homomorphisms.

Let \(X\) be a finite \(\Gamma\)-set, let \(n\ge 2\), and set
\(R=\Z/n\Z\) and \(A=R[X]=\bigoplus_{x\in X}Rv_x\). Choose a primitive
\(n\)-th root of unity \(q\in\kk^\times\).
For \(x\in X\) and \(P\in\binom X2\), write
\(\Gamma_x=\{g\in\Gamma\mid gx=x\}\) and
\(\Gamma_P=\{g\in\Gamma\mid gP=P\}\). Let
\(\varepsilon_P:\Gamma_P\to\{\pm1\}\) be the sign of the action on
\(P\), and write \(R_{\varepsilon_P}\) for the resulting
\(\Gamma_P\)-module.

The choice of \(q\) identifies bicharacters on \(A\) with
\(R\)-bilinear maps \(B:A\times A\to R\).
Since \(A\) is free on the basis \(\{v_x\}_{x\in X}\), such a map is
determined by the matrix entries \(B_{xy}=B(v_x,v_y)\), for \(x,y\in X\).
Thus \(\Bich(A)\cong R^{X\times X}\).
The corresponding bicharacter is
\[
b_B\!\left(\sum_{x\in X}a_xv_x,\sum_{y\in X}c_yv_y\right)
=
q^{\sum_{x,y\in X}a_xc_yB_{xy}}.
\]

The alternating and symmetric parts are
\[
\AltBich(A)
=
\{B\in R^{X\times X}
\mid
B_{yx}=-B_{xy},\;
B_{xx}=0\},
\]
and
\[
\SymBich(A)
=
\{C\in R^{X\times X}
\mid
C_{yx}=C_{xy}\}.
\]

The action on \(R^{X\times X}\) is given by
\((g\cdot B)_{xy}=B_{g^{-1}x,g^{-1}y}\).

\begin{proposition}\label{prop:perm-decomposition}
There are isomorphisms of \(\Gamma\)-modules
\[
\SymBich(A)
\cong
R[X]
\oplus
R\!\left[\binom{X}{2}\right]
\]
and
\[
\AltBich(A)
\cong
\bigoplus_{P\in\Gamma\backslash\binom{X}{2}}
\Ind_{\Gamma_P}^{\Gamma}R_{\varepsilon_P},
\]
where \(\Ind_H^\Gamma\) denotes induction of modules.
\end{proposition}

\begin{proof}
A symmetric form \(C\) is determined by the coefficients
\(C(v_x,v_x)\), indexed by \(X\), and \(C(v_x,v_y)\), indexed
independently of order by \(\binom{X}{2}\). These assignments are
\(\Gamma\)-equivariant and give
\(\SymBich(A)\cong R[X]\oplus R[\binom{X}{2}]\).

For an alternating form, the coefficient line associated with an
unordered pair \(P=\{x,y\}\) is generated by the orientation
\((x,y)\), with \((y,x)=-(x,y)\). Thus \(\Gamma_P\) acts on that line
through \(\varepsilon_P\), and the orbit of \(P\) contributes
\(\Ind_{\Gamma_P}^{\Gamma}R_{\varepsilon_P}\). Summing over the
\(\Gamma\)-orbits gives
\(\AltBich(A)\cong\bigoplus_P
\Ind_{\Gamma_P}^{\Gamma}R_{\varepsilon_P}\), with \(P\) ranging over
\(\Gamma\backslash\binom{X}{2}\).
\end{proof}

Proposition~\ref{prop:perm-decomposition} can be sharpened by keeping
track of the exact sequence relating symmetric, arbitrary, and
alternating bicharacters.

For each chosen orbit representative \(P=\{x,y\}\), fix an orientation
\((x,y)\) and let
\(\operatorname{Or}(P)=\{(x,y),(y,x)\}\) be the set of orientations of
\(P\). We use the exact sequence
\begin{equation}\label{eq:orientation-sequence}
0
\longrightarrow
R
\longrightarrow
R[\operatorname{Or}(P)]
\longrightarrow
R_{\varepsilon_P}
\longrightarrow
0,
\end{equation}
where \(R\) has the trivial \(\Gamma_P\)-action. In
\eqref{eq:orientation-sequence}, the map from \(R\) sends
\(1\longmapsto (x,y)+(y,x)\), while the map to
\(R_{\varepsilon_P}\) sends \((x,y)\longmapsto 1\) and
\((y,x)\longmapsto -1\).
Reversing the chosen orientation changes this quotient identification,
and hence the corresponding coordinate of the connecting map, by a
sign. The global connecting map is canonical, while the kernel and the
vanishing of each coordinate are independent of the choice.

\begin{proposition}\label{prop:perm-bich-exact}
There is an isomorphism of \(\Gamma\)-modules
\[
\Bich(A)
\cong
R[X]
\oplus
\bigoplus_{P\in\Gamma\backslash\binom{X}{2}}
\Ind_{\Gamma_P}^{\Gamma}R[\operatorname{Or}(P)].
\]
Under these decompositions, the exact sequence
\[
0
\longrightarrow
\SymBich(A)
\longrightarrow
\Bich(A)
\xrightarrow{\Alt}
\AltBich(A)
\longrightarrow
0
\]
is the direct sum of the diagonal sequence
\[
0\longrightarrow R[X]\xrightarrow{=}R[X]\longrightarrow0\longrightarrow0
\]
and the sequences induced from \eqref{eq:orientation-sequence}.
\end{proposition}

\begin{proof}
The diagonal entries give the summand \(R[X]\), while the off-diagonal
entries in the orbit of \(P\) give
\(\Ind_{\Gamma_P}^{\Gamma}R[\operatorname{Or}(P)]\). On the local
orientation module, the symmetric coefficient
\((x,y)+(y,x)\) generates the kernel of \(\Alt\), and the quotient is
the sign module; this is precisely the induced form of
\eqref{eq:orientation-sequence}. Taking the direct sum of these local
sequences over the \(\Gamma\)-orbits gives the exact sequence relating
\(\SymBich(A)\), \(\Bich(A)\), and \(\AltBich(A)\).
\end{proof}

For a subgroup \(H\leq\Gamma\) and an \(H\)-module \(M\), Shapiro's
lemma for group cohomology gives a natural isomorphism
\[
H^r\!\left(\Gamma,\operatorname{Coind}_H^\Gamma M\right)
\cong H^r(H,M),
\qquad r\geq0.
\]
See~\cite[Chapter~III]{Brown1982}. When \([\Gamma:H]\) is finite, the
induced and coinduced modules are naturally isomorphic. Since \(X\) is
finite, the point and pair stabilizers in
Proposition~\ref{prop:perm-decomposition} have finite index, so the
Shapiro isomorphism applies to all the induced summands. Its naturality
in \(M\) also gives compatibility with the connecting maps of short
exact sequences.

More explicitly, the permutation summands decompose as
\[
R[X]\cong
\bigoplus_{x\in\Gamma\backslash X}\Ind_{\Gamma_x}^{\Gamma}R,
\qquad
R\!\left[\binom X2\right]\cong
\bigoplus_{P\in\Gamma\backslash\binom X2}
\Ind_{\Gamma_P}^{\Gamma}R.
\]

Applying the finite-index Shapiro isomorphism to
Proposition~\ref{prop:perm-decomposition} gives the first-cohomology and
obstruction groups in Corollary~\ref{cor:perm-shapiro}.

\begin{corollary}\label{cor:perm-shapiro}
There are isomorphisms
\[
H^2(\Gamma,\SymBich(A))
\cong
\bigoplus_{x\in\Gamma\backslash X}
H^2(\Gamma_x,R)
\oplus
\bigoplus_{P\in\Gamma\backslash\binom{X}{2}}
H^2(\Gamma_P,R)
\]
and
\[
H^1(\Gamma,\AltBich(A))
\cong
\bigoplus_{P\in\Gamma\backslash\binom{X}{2}}
H^1(\Gamma_P,R_{\varepsilon_P}).
\]
\end{corollary}

\begin{proof}
For each point orbit, Proposition~\ref{prop:perm-decomposition}
contributes \(\Ind_{\Gamma_x}^{\Gamma}R\); for each pair orbit, it
contributes \(\Ind_{\Gamma_P}^{\Gamma}R\) to the symmetric module and
\(\Ind_{\Gamma_P}^{\Gamma}R_{\varepsilon_P}\) to the alternating
module. Applying Shapiro's isomorphism to the point- and pair-orbit
summands gives both formulas.
\end{proof}

The \(H^1\)-isomorphism in Corollary~\ref{cor:perm-shapiro} identifies
crossed homomorphisms modulo principal ones. The cocycles themselves fit
into the exact sequence
\[
\begin{gathered}
0
\longrightarrow
\AltBich(A)/\AltBich(A)^\Gamma
\longrightarrow
Z^1(\Gamma,\AltBich(A))
\\
\longrightarrow
\bigoplus_{P\in\Gamma\backslash\binom{X}{2}}
H^1(\Gamma_P,R_{\varepsilon_P})
\longrightarrow
0.
\end{gathered}
\]
The map from \(\AltBich(A)/\AltBich(A)^\Gamma\) sends the class of
\(B\) to the principal cocycle \(g\mapsto g\cdot B-B\). The map to the
direct sum of stabilizer cohomology groups is the quotient
\(Z^1\to H^1\) followed by the Shapiro isomorphism.

\begin{remark}\label{rem:perm-bich-obstruction}
Let
\[
\beta_P\colon H^1(\Gamma_P,R_{\varepsilon_P})
\longrightarrow H^2(\Gamma_P,R)
\]
be the connecting map associated with
\eqref{eq:orientation-sequence}. Under the isomorphisms of
Corollary~\ref{cor:perm-shapiro}, the bicharacter obstruction on
\(H^1\) is the direct sum of the maps \(\beta_P\), with zero components
in the diagonal summands \(H^2(\Gamma_x,R)\). Consequently, its kernel is
\[
\bigoplus_{P\in\Gamma\backslash\binom X2}\ker(\beta_P),
\]
and the bicharacter-liftable crossed homomorphisms are the inverse image
of this subgroup under \(Z^1\to H^1\). The map \(\beta_P\) vanishes if
\(\varepsilon_P\) is trivial. It also vanishes when \(2\in R^\times\),
because then \eqref{eq:orientation-sequence} has the equivariant section
\(1\mapsto \frac12((x,y)-(y,x))\). In particular, all the maps
\(\beta_P\) vanish for odd \(n\), in agreement with
Corollary~\ref{cor:odd-exponent}.
\end{remark}

\subsection{The inversion cocycle for the symmetric group}
\label{sec:symmetric}

Assume \(m\ge 2\). We specialize to \(\Gamma=\mathfrak S_m\) acting on
\(X=\{1,\dots,m\}\), and set \(A=R[X]\).
We use the convention that \(\mathfrak S_0\) and \(\mathfrak S_1\) are
trivial groups.

The symmetric coefficient module becomes
\[
\SymBich(A)
\cong
\Ind_{\mathfrak S_{m-1}}^{\mathfrak S_m}R
\oplus
\Ind_{\mathfrak S_2\times \mathfrak S_{m-2}}^{\mathfrak S_m}R,
\]
and hence
\[
H^2(\mathfrak S_m,\SymBich(A))
\cong
H^2(\mathfrak S_{m-1},R)
\oplus
H^2(\mathfrak S_2\times \mathfrak S_{m-2},R).
\]
This is the target of
\(\operatorname{ob}_{\operatorname{bich}}\) in
Theorem~\ref{thm:bich-obstruction}.

Let \(E_{ij}\), \(i\neq j\), denote the alternating matrix determined
by
\[
(E_{ij})_{ij}=1,
\qquad
(E_{ij})_{ji}=-1,
\qquad
E_{ji}=-E_{ij},
\qquad
(E_{ij})_{rs}=0
\quad\text{otherwise}.
\]

Since \(\mathfrak S_m\) acts transitively on unordered pairs, the
stabilizer of \(\{1,2\}\) is
\(\mathfrak S_2\times \mathfrak S_{m-2}\),
and the factor \(\mathfrak S_2\) acts by the sign character on
\(RE_{12}\).
Write \(R_\varepsilon=R_{\varepsilon_{\{1,2\}}}\).
Hence
\[
\AltBich(A)
\cong
\Ind_{\mathfrak S_2\times \mathfrak S_{m-2}}^{\mathfrak S_m}R_\varepsilon.
\]

The symmetric-group action gives a natural crossed homomorphism built
from inversion sets.

\begin{definition}
Let \(d_{\operatorname{inv}}\colon \mathfrak S_m\longrightarrow
\AltBich(A)\) be given on the generators \(s_i=(i,i+1)\), \(1\le i<m\),
by \(d_{\operatorname{inv}}(s_i)=E_{i,i+1}\).
Equivalently,
\[
d_{\operatorname{inv}}(w)
=
\sum_{\substack{a<b\\ w(a)>w(b)}}
E_{w(b),w(a)}
=
\sum_{\substack{r<s\\ w^{-1}(r)>w^{-1}(s)}}E_{rs}.
\]
\end{definition}
The signed inversion-set identity, obtained by tracking the pairs whose
orientation is reversed first by \(v\) and then by \(w\), gives
\[
d_{\operatorname{inv}}(wv)
=
d_{\operatorname{inv}}(w)+w\cdot d_{\operatorname{inv}}(v),
\]
so \(d_{\operatorname{inv}}\in
Z^1(\mathfrak S_m,\AltBich(A))\).

Write \(R[2]=\{r\in R\mid 2r=0\}\).

\begin{proposition}\label{prop:symmetric-alt-cohomology}
The Shapiro isomorphism gives
\[
H^1(\mathfrak S_m,\AltBich(A))
\cong
H^1(\mathfrak S_2\times \mathfrak S_{m-2},R_\varepsilon)
\cong
\begin{cases}
R/2R, & m=2,3,\\[2mm]
R/2R\oplus R[2], & m\ge 4.
\end{cases}
\]
Under this isomorphism, the class of \(d_{\operatorname{inv}}\)
corresponds to \(1\in R/2R\) and to \(0\) in the \(R[2]\)-summand when
\(m\ge 4\).
\end{proposition}

\begin{proof}
Let \(H=\mathfrak S_2\times \mathfrak S_{m-2}\) be the stabilizer of
\(\{1,2\}\), and let \(t\) be the nontrivial
element of the first factor. The first factor acts on \(R_\varepsilon\)
by \(-1\), while \(\mathfrak S_{m-2}\) acts trivially.

A crossed homomorphism \(c\) from \(H\) to \(R_\varepsilon\) is determined by
\(a=c(t)\) and by its restriction \(\phi\) to \(\mathfrak S_{m-2}\). The
restriction \(\phi\) is a homomorphism into \(R\), because
\(\mathfrak S_{m-2}\) acts trivially. The commutation relation between
\(t\) and \(\mathfrak S_{m-2}\) imposes \(2\phi(h)=0\) for every
\(h\in \mathfrak S_{m-2}\), so
\(\phi\in\Hom(\mathfrak S_{m-2}^{\operatorname{ab}},R[2])\), where
\({\operatorname{ab}}\) denotes abelianization.
Conversely, given \(a\in R\) and such a homomorphism \(\phi\), the
formulas \(c(h)=\phi(h)\) and \(c(th)=a-\phi(h)\) define a crossed
homomorphism \(H\to R_\varepsilon\).
Principal crossed homomorphisms replace \(a\) by \(a-2u\) and do not
change \(\phi\). Therefore
\[
H^1(H,R_\varepsilon)
\cong
R/2R
\oplus
\Hom(\mathfrak S_{m-2}^{\operatorname{ab}},R[2]).
\]
Since \(\mathfrak S_r^{\operatorname{ab}}=0\) for \(r=0,1\) and
\(\mathfrak S_r^{\operatorname{ab}}\cong \Z/2\Z\) for \(r\ge 2\),
the Hom group is zero for \(m=2,3\) and isomorphic to \(R[2]\) for
\(m\ge4\).

For \(d_{\operatorname{inv}}\), using the evaluation-at-\((1,2)\)
realization of the Shapiro map, the restriction to the stabilizer has
coefficient \(1\) on the oriented pair \((1,2)\) when evaluated at
\(t=(1,2)\), and it has zero coefficient on \((1,2)\) for elements of
the \(\mathfrak S_{m-2}\)-factor. Hence its \(R/2R\)-coordinate is
\(1\); when \(m\geq4\), its \(R[2]\)-coordinate is zero.
\end{proof}

For \(R=\Z/n\Z\), Proposition~\ref{prop:symmetric-alt-cohomology}
shows that \(H^1(\mathfrak S_m,\AltBich(A))\) vanishes when \(n\) is
odd. When \(n\) is even, this group is \(\Z/2\Z\) for \(m=2,3\) and
\((\Z/2\Z)^2\) for \(m\ge4\); the class of
\(d_{\operatorname{inv}}\) generates the first factor.

Assume that \(n\) is odd. Let
\[
\Omega=\sum_{i<j}E_{ij},
\qquad
\eta=-2^{-1}\Omega.
\]

For an unordered pair \(\{r,s\}\), \(r<s\), the coefficient of
\(E_{rs}\) in \(w\cdot\Omega-\Omega\) is zero unless
\(\{r,s\}=\{w(a),w(b)\}\) for an inversion \(a<b\), \(w(a)>w(b)\). In
that case the contribution is \(-2E_{w(b),w(a)}\). Hence
\[
w\cdot\Omega-\Omega
=
-2\,d_{\operatorname{inv}}(w).
\]
Therefore
\[
d_{\operatorname{inv}}(w)
=
w\cdot\eta-\eta,
\]
and \(d_{\operatorname{inv}}\) is principal. Thus, for odd \(n\),
\(d_{\operatorname{inv}}\) has trivial class in
\(H^1(\mathfrak S_m,\AltBich(A))\). For \(n>1\), however, it is a
nonzero element of \(Z^1(\mathfrak S_m,\AltBich(A))\). The equivalence
defining \(\Ecal(A,\mathfrak S_m)\) in~\eqref{eq:E} retains this
\(Z^1\)-datum even when its \(H^1\)-class vanishes.

Proposition~\ref{prop:inversion-lift} now determines whether this
crossed homomorphism, rather than merely its cohomology class, admits a
bicharacter lift.

\begin{proposition}\label{prop:inversion-lift}
\(d_{\operatorname{inv}}\) admits a bicharacter lift if and only if
\(n\) is odd.
\end{proposition}
\begin{proof}
If \(n\) is odd, then \(2\) is invertible in \(R=\Z/n\Z\). The
canonical half-alternating construction of
Corollary~\ref{cor:odd-exponent} therefore lifts every alternating
bicharacter, and in particular lifts \(d_{\operatorname{inv}}\).

Conversely, suppose that a bicharacter lift \(\widetilde d\) of
\(d_{\operatorname{inv}}\) exists. For an adjacent transposition \(s_i\),
write \(B_i=\widetilde d(s_i)\). The involutivity relation gives
\(B_i+s_i\cdot B_i=0\). Evaluating at \((v_i,v_{i+1})\) yields
\(B_i(v_i,v_{i+1})+B_i(v_{i+1},v_i)=0\). On the other hand,
\(\Alt(B_i)=E_{i,i+1}\) gives
\(B_i(v_i,v_{i+1})-B_i(v_{i+1},v_i)=1\).
Adding these equations gives
\[
2B_i(v_i,v_{i+1})=1
\]
in \(\Z/n\Z\). Hence \(2\) has a multiplicative inverse in \(\Z/n\Z\),
which happens precisely when \(n\) is odd. Therefore no bicharacter lift
of \(d_{\operatorname{inv}}\) exists when \(n\) is even.
\end{proof}

\begin{proposition}\label{prop:inversion-strict-lift}
For every \(m\ge2\) and \(n\ge2\), the cohomological datum
\(\vartheta_*d_{\operatorname{inv}}\) admits a strict lift.
Consequently, if \(n\) is even, then
\[
\operatorname{ob}_{\operatorname{bich}}(d_{\operatorname{inv}})\ne0,
\qquad
H^2(\mathfrak S_m,\iota)
\bigl(\operatorname{ob}_{\operatorname{bich}}
(d_{\operatorname{inv}})\bigr)
=\delta(\vartheta_*d_{\operatorname{inv}})=0.
\]
\end{proposition}

\begin{proof}
Fix a normalized set-theoretic section
\(R=\Z/n\Z\to\Z\), used in every coordinate, and denote the chosen
representative of \(r\in R\) by \(\bar r\). Choose
\(\zeta\in\kk^\times\) with \(\zeta^2=q\). For
\(a,b\in A=R^m\) and \(1\le i<m\), define, with exponents of \(q\)
read in \(R\),
\[
c_i(a,b)=q^{a_i b_{i+1}},
\qquad
f_i(a)=\zeta^{\bar a_i\bar a_{i+1}},
\qquad
\tau_i=c_i\,\partial f_i.
\]
The function \(c_i\) is a normalized bicharacter and \(\partial f_i\)
is a normalized \(2\)-coboundary. Hence \(\tau_i\) is a normalized
\(2\)-cocycle. Moreover,
\[
\Alt(c_i)(a,b)=q^{a_i b_{i+1}-a_{i+1}b_i},
\qquad
[\tau_i]=[c_i]
=\vartheta(E_{i,i+1})
=\vartheta(d_{\operatorname{inv}}(s_i)).
\]

Put
\(\sigma_i(a,b)=
q^{a_i b_{i+1}+a_{i+1}b_i}\).
Since the same section is used in every coordinate,
\[
c_i(s_i\cdot c_i)=\sigma_i,
\qquad
s_i\cdot f_i=f_i,
\qquad
\partial(f_i^2)=\sigma_i^{-1}.
\]
In the identity \(\partial(f_i^2)=\sigma_i^{-1}\), carries in the chosen
integer representatives contribute only multiples of \(n\) to the
exponent of \(q\). It follows that
\(\tau_i(s_i\cdot\tau_i)=1\).

For the remaining Coxeter relations, write
\(c_{rs}(a,b)=q^{a_r b_s}\) and
\(f_{rs}(a)=\zeta^{\bar a_r\bar a_s}=f_{sr}(a)\). Then
\(w\cdot c_{rs}=c_{w(r),w(s)}\) and
\(w\cdot f_{rs}=f_{w(r),w(s)}\). The commutation relations for
\(|i-j|>1\) follow immediately. In the braid relation, the two sides
of
\[
\tau_i(s_i\cdot\tau_{i+1})(s_is_{i+1}\cdot\tau_i)
=
\tau_{i+1}(s_{i+1}\cdot\tau_i)
(s_{i+1}s_i\cdot\tau_{i+1})
\]
have the same bicharacter factor
\(c_{i,i+1}c_{i,i+2}c_{i+1,i+2}\) and the same coboundary factor
\(\partial(f_{i,i+1}f_{i,i+2}f_{i+1,i+2})\).

The multiplication in
\(Z^2(A,\kk^\times)\rtimes\mathfrak S_m\) is
\((\omega,w)(\nu,v)=(\omega(w\cdot\nu),wv)\). Thus
\(s_i\mapsto(\tau_i,s_i)\) respects the Coxeter presentation. Writing the resulting
homomorphism as \(w\mapsto(\tau_w,w)\) gives
\(\tau_{wv}=\tau_w(w\cdot\tau_v)\). The crossed homomorphism
\(w\mapsto[\tau_w]\) agrees with
\(\vartheta_*d_{\operatorname{inv}}\) on every generator and hence
equals it. This proves the strict-lift assertion. For even \(n\), the
nonvanishing assertion follows from Proposition~\ref{prop:inversion-lift}
and Theorem~\ref{thm:bich-obstruction}; the vanishing assertion follows
from the constructed strict lift and
Theorem~\ref{thm:general-obstruction}. Their comparison is
\eqref{eq:bich-strict-comparison}.
\end{proof}

\begin{example}
Assume that \(n\) is odd, let \(\Gamma=\mathfrak S_2=\langle s\rangle\),
set \(R=\Z/n\Z\), and take \(A=R v_1\oplus R v_2\). Let \(s\)
interchange \(v_1\) and \(v_2\). Let \(q\in\kk^\times\)
be a primitive \(n\)-th root of unity and let \(h=2^{-1}\in R\).
The canonical lift of \(d_{\operatorname{inv}}\) is
\[
b_1=1,
\qquad
b_s(\alpha,\beta)
=
q^{h(\alpha_1\beta_2-\alpha_2\beta_1)},
\]
where \(\alpha=\alpha_1v_1+\alpha_2v_2\) and
\(\beta=\beta_1v_1+\beta_2v_2\). Write \(e_a\) for the idempotent of
\(\kk^A\) corresponding to \(a=(a_1,a_2)\in R^2\). Then
\(H(\mathfrak S_2,A,b)\) has dimension \(2n^2\), multiplication
\[
(e_a u_s)(e_c u_s)
=
\delta_{a,s\cdot c}\,e_a,
\]
and coproduct on the nontrivial homogeneous component
\[
\Delta(e_a u_s)
=
\sum_{\alpha+\beta=a}
q^{h(\alpha_1\beta_2-\alpha_2\beta_1)}
(e_\alpha u_s)\otimes(e_\beta u_s).
\]
The component over \(1\in \mathfrak S_2\) has the usual function-algebra
coproduct.
\end{example}

\subsection{The full extension group for permutation modules}
\label{sec:permutation-full-opext}

The group \(\Ecal(A,\Gamma)\) in~\eqref{eq:E} records the
algebra-split part of the extension problem. To compare it with all
cocentral abelian extensions for the fixed action and trivial
coaction, write
\[
\Opext_\Gamma(A)
=
\Opext(\kk\Gamma,\kk^A;\mathord{\cdot},\rho_{\mathrm{triv}}).
\]
The normalized total complex for abelian extensions
specializes in this setting as follows; compare
\cite[Theorem~4.2]{GalindoMorales2019} and
\cite{Schauenburg2002Kac,GrunenfelderMastnak2004}.
No finiteness assumption on \(\Gamma\) is needed in this subsection.
A total \(2\)-cocycle is represented by
\[
\sigma\in C^2\bigl(\Gamma,C^1(A,\kk^\times)\bigr),
\qquad
\tau\in C^1\bigl(\Gamma,Z^2(A,\kk^\times)\bigr).
\]
If
\[
\Omega_\tau(g,h)
=
\tau_g(g\cdot\tau_h)\tau_{gh}^{-1},
\]
the total-cocycle identities are
\begin{equation}\label{eq:full-total-cocycle}
\partial_\Gamma\sigma=1,
\qquad
\partial\sigma_{g,h}=\Omega_\tau(g,h)^{-1}.
\end{equation}
For a normalized
\(f\in C^1(\Gamma,C^1(A,\kk^\times))\), the total-coboundary
equivalence is
\begin{equation}\label{eq:full-total-equivalence}
(\sigma,\tau)
\sim
\bigl(\sigma\,\partial_\Gamma f,
      \tau\,(\partial f)^{-1}\bigr).
\end{equation}

Set \(A^\vee=\Hom(A,\kk^\times)\), and let
\[
\partial_{\mathrm{char}}:
H^2\bigl(\Gamma,B^2(A,\kk^\times)\bigr)
\longrightarrow
H^3(\Gamma,A^\vee)
\]
be the connecting map associated with
\[
1\longrightarrow A^\vee
\longrightarrow C^1(A,\kk^\times)
\xrightarrow{\partial}B^2(A,\kk^\times)
\longrightarrow1.
\]
For a total cocycle \((\sigma,\tau)\), consider the assignments
\[
\begin{aligned}
\operatorname{cl}(\sigma,\tau)(g)
&=[\tau_g]\in H^2(A,\kk^\times),\\
\operatorname{mult}(\sigma,\tau)
&=[\sigma]\in H^2\bigl(\Gamma,C^1(A,\kk^\times)\bigr).
\end{aligned}
\]

\begin{proposition}\label{prop:full-kac-comparison}
These assignments descend to extension classes. Moreover, there is a
low-degree exact sequence
\begin{equation}\label{eq:kac-low-degree}
0
\longrightarrow
H^2(\Gamma,A^\vee)
\longrightarrow
\Opext_\Gamma(A)
\xrightarrow{\operatorname{cl}}
Z^1\bigl(\Gamma,H^2(A,\kk^\times)\bigr)
\xrightarrow{d_2}
H^3(\Gamma,A^\vee).
\end{equation}
The map \(H^2(\Gamma,A^\vee)\to\Opext_\Gamma(A)\) sends \([\chi]\) to
\([(\chi,1)]\), and
\(\Ecal(A,\Gamma)\cong\ker(\operatorname{mult})\).
For every \(d\in Z^1(\Gamma,H^2(A,\kk^\times))\), the strict obstruction of
Theorem~\ref{thm:general-obstruction} satisfies
\begin{equation}\label{eq:kac-strict-comparison}
d_2(d)
=
\partial_{\mathrm{char}}(\delta(d)).
\end{equation}
\end{proposition}

\begin{proof}
The identity \(\partial\sigma_{g,h}=\Omega_\tau(g,h)^{-1}\)
in~\eqref{eq:full-total-cocycle} shows that \(g\mapsto[\tau_g]\) is
crossed, and
\eqref{eq:full-total-equivalence} does not change its values in
cohomology. If these values are trivial, a total coboundary makes
\(\tau=1\); the remaining \(\sigma\) is then an \(A^\vee\)-valued
\(2\)-cocycle. The remaining equivalences are precisely
\(A^\vee\)-valued \(2\)-coboundaries. This proves injectivity of the map
from \(H^2(\Gamma,A^\vee)\) and exactness at \(\Opext_\Gamma(A)\)
in~\eqref{eq:kac-low-degree}.
The target is \(Z^1\), rather than \(H^1\), because
\eqref{eq:full-total-equivalence} changes each \(\tau_g\) by a
group \(2\)-coboundary on \(A\) but does not identify the resulting crossed
homomorphism \(g\mapsto[\tau_g]\) modulo principal \(\Gamma\)-cocycles.

For
\(d\in Z^1(\Gamma,H^2(A,\kk^\times))\), choose representatives
\(\widetilde\tau_g\) and let
\(\Omega_d(g,h)=\widetilde\tau_g
(g\cdot\widetilde\tau_h)\widetilde\tau_{gh}^{-1}\), as in
Theorem~\ref{thm:general-obstruction}. Choose
\(a\in C^2(\Gamma,C^1(A,\kk^\times))\) with
\(\partial a=\Omega_d\). Then \(\partial_\Gamma a\) is
\(A^\vee\)-valued and represents
\(\partial_{\mathrm{char}}(\delta(d))\). By
\cite[formula~(4-1)]{GalindoMorales2019}, \(d_2(d)\) is represented by
\(\partial_\Gamma a\). The total-cocycle correction is instead
\(a^{-1}\). The pair
\((a^{-1},\widetilde\tau)\) satisfies the mixed identity in
\eqref{eq:full-total-cocycle}, while its horizontal defect is
\((\partial_\Gamma a)^{-1}\), which represents \(d_2(d)^{-1}\). If
\(\widetilde\tau\) is replaced by
\(\widetilde\tau\,\partial f\), one may replace \(a\) by
\(a\,\partial_\Gamma f\), which leaves \(\partial_\Gamma a\)
unchanged. A second choice of \(a\) differs from the first by an
\(A^\vee\)-valued \(2\)-cochain and therefore changes the horizontal
defect by a \(3\)-coboundary. Thus its class is independent of both
choices. Taking product representatives and the product of the
corresponding cochains \(a\) shows that \(d_2\) is a homomorphism. The
horizontal defect can be removed by an \(A^\vee\)-valued
\(2\)-cochain exactly when \(d_2(d)\) is trivial. This proves exactness at
\(Z^1(\Gamma,H^2(A,\kk^\times))\) in~\eqref{eq:kac-low-degree}, as well
as~\eqref{eq:kac-strict-comparison}.

Finally, \(\partial_\Gamma\sigma=1\) makes \([\sigma]\) a
cohomology class, and~\eqref{eq:full-total-equivalence} changes it by
a coboundary. If this class is trivial, the same equivalence makes
\(\sigma=1\), and the mixed identity makes \(\tau\) a strict Hopf
datum. Equivalences between representatives with \(\sigma=1\) have
\(\partial_\Gamma f=1\); after replacing \(f\) by \(f^{-1}\), this
is exactly the relation in Definition~\ref{def:equivalent-strict-data}.
Thus \(\ker(\operatorname{mult})\simeq\Ecal(A,\Gamma)\).
\end{proof}

For \(d\in Z^1(\Gamma,\AltBich(A))\), the two obstruction
comparisons combine to give
\begin{align*}
H^2(\Gamma,\iota)
\bigl(\operatorname{ob}_{\operatorname{bich}}(d)\bigr)
&=\delta(\vartheta_*d),\\
\partial_{\mathrm{char}}\bigl(\delta(\vartheta_*d)\bigr)
&=d_2(\vartheta_*d).
\end{align*}
Thus a bicharacter lift of \(d\), a strict lift of \(\vartheta_*d\), and
a realization of \(\vartheta_*d\) in \(\Opext_\Gamma(A)\) exist precisely
when, respectively,
\(\operatorname{ob}_{\operatorname{bich}}(d)\),
\(\delta(\vartheta_*d)\), and \(d_2(\vartheta_*d)\) vanish.

We retain from Section~\ref{sec:permutation} the notation
\(A=R[X]\) and \(R=\Z/n\Z\). The primitive root \(q\) identifies
alternating bicharacters with alternating matrices over \(R\).

\begin{theorem}\label{thm:permutation-kac-vanishing}
For every finite permutation module \(A=R[X]\), the transgression
\(d_2\circ\vartheta_*\) is trivial. More precisely,
\begin{equation}\label{eq:permutation-split-kac}
0
\longrightarrow
H^2(\Gamma,A^\vee)
\longrightarrow
\Opext_\Gamma(A)
\xrightarrow{\vartheta_*^{-1}\circ\operatorname{cl}}
Z^1\bigl(\Gamma,\AltBich(A)\bigr)
\longrightarrow0
\end{equation}
admits a section that is a group homomorphism. Consequently,
\begin{equation}\label{eq:permutation-full-opext}
\Opext_\Gamma(R[X])
\cong
H^2(\Gamma,R[X])
\oplus
Z^1\bigl(\Gamma,\AltBich(R[X])\bigr),
\end{equation}
where the first summand uses the identification
\(A^\vee\simeq R[X]\) determined by \(q\).
\end{theorem}

\begin{proof}
Fix a temporary ordering of \(X\). For an alternating matrix \(B\), let
\[
U(B)_{xy}
=
\begin{cases}
B_{xy},&x<y,\\
0,&x\geq y.
\end{cases}
\]
Then \(U(B)-U(B)^{\mathsf T}=B\).
Given \(d\in Z^1(\Gamma,\AltBich(A))\), define
\[
\tau^d_g(\alpha,\beta)
=
q^{\alpha^{\mathsf T}U(d_g)\beta}.
\]
The alternating part of \(\tau^d_g\) is \(d_g\). Its crossed defect
has matrix
\begin{equation}\label{eq:C-d-matrix}
C_d(g,h)
=
U(d_g)+g\cdot U(d_h)-U(d_{gh}).
\end{equation}
Since \(d_{gh}=d_g+g\cdot d_h\), subtracting the transpose in
\eqref{eq:C-d-matrix} gives zero. Each summand in that equation has
zero diagonal, so \(C_d(g,h)\) is symmetric with zero diagonal.

For every such matrix \(C\), put
\begin{equation}\label{eq:sigma-d-quadratic}
\mathfrak q_C(\alpha)
=
q^{\sum_{x<y}C_{xy}\alpha_x\alpha_y},
\qquad
\sigma^d_{g,h}=\mathfrak q_{C_d(g,h)}.
\end{equation}
Because \(C\) is symmetric with zero diagonal, the exponent is an
intrinsic sum over unordered pairs. Hence \(C\mapsto\mathfrak q_C\) is
independent of the temporary ordering of \(X\) and is
\(\Gamma\)-equivariant.
Direct expansion gives
\[
(\partial\mathfrak q_C)(\alpha,\beta)
=q^{-\alpha^{\mathsf T}C\beta}.
\]
Hence \(\partial\sigma^d_{g,h}=\Omega_{\tau^d}(g,h)^{-1}\).
Moreover,
\(C_d=\partial_\Gamma(U\circ d)\). The assignment
\(C\mapsto\mathfrak q_C\) is additive, so
\(\partial_\Gamma\sigma^d
=\mathfrak q_{\partial_\Gamma C_d}=1\).
Thus \((\sigma^d,\tau^d)\) is a total cocycle mapping to
\(\vartheta_*d\).
At the level of obstruction classes, the identity
\(\partial\sigma^d=\Omega_{\tau^d}^{-1}\) reads
\[
H^2(\Gamma,\partial)[\sigma^d]
=
\delta(\vartheta_*d)^{-1}.
\]

All assignments in the construction are additive, so
\(d\mapsto[(\sigma^d,\tau^d)]\) is a homomorphic section. Its class
does not depend on the temporary order. Indeed, if \(U'\) is obtained
from another order and \(D=U'-U\), then every \(D(B)\) is symmetric
with zero diagonal. For
\(f_g=\mathfrak q_{D(d_g)}\), the two constructions satisfy
\[
(\sigma^{U'},\tau^{U'})
=
\bigl(\sigma^U\partial_\Gamma f,
      \tau^U(\partial f)^{-1}\bigr),
\]
and hence are equivalent by~\eqref{eq:full-total-equivalence}.
If \(q'=q^r\), with \(r\in R^\times\), is another primitive root, the
matrix representing a fixed bicharacter changes from \(B\) to
\(r^{-1}B\). The resulting functions are therefore unchanged. The
split exact sequence~\eqref{eq:permutation-split-kac} now gives the
direct-sum decomposition, and \(q\) identifies \(A^\vee\) with \(R[X]\).
\end{proof}

The splitting also identifies the algebra-split subgroup. Let
\(\jmath:H^2(\Gamma,A^\vee)\to
H^2(\Gamma,C^1(A,\kk^\times))\) be induced by
\(A^\vee\hookrightarrow C^1(A,\kk^\times)\). Under the
decomposition~\eqref{eq:permutation-full-opext}, the multiplication
component of \((\xi,d)\) is \(\jmath(\xi)[\sigma^d]\). Thus
\(\Ecal(A,\Gamma)\) corresponds to
\[
\bigl\{(\xi,d)\mid \jmath(\xi)[\sigma^d]=1\bigr\}.
\]
The projection of this subgroup onto the second factor has kernel
\(\ker\jmath\) and image \(\ker(\delta\circ\vartheta_*)\), recovering
Corollary~\ref{cor:general-obstruction-exact}.

Although the section in Theorem~\ref{thm:permutation-kac-vanishing} is
well defined on extension classes, the coordinate decomposition in
Corollary~\ref{cor:full-opext-orbit-formula} depends on the choices made
in that corollary and is noncanonical.

Thus every crossed cohomological datum \(d\) for a permutation module is
realized by a cocentral abelian extension. By
Theorem~\ref{thm:general-obstruction}, it has an algebra-split
realization exactly when \(\delta(\vartheta_*d)\) is trivial. The quadratic
correction compensates for the strict obstruction in the total complex
by canceling the transgression \(d_2(\vartheta_*d)\) without forcing
the strict obstruction to vanish.

\begin{corollary}\label{cor:full-opext-orbit-formula}
Choose one representative \(x\) of each \(\Gamma\)-orbit in \(X\) and
one representative \(P\) of each \(\Gamma\)-orbit in
\(\binom{X}{2}\). Then there is a noncanonical isomorphism
\begin{equation}\label{eq:full-opext-orbit-formula}
\begin{aligned}
\Opext_\Gamma(R[X])
\cong{}&
\bigoplus_{x\in\Gamma\backslash X}H^2(\Gamma_x,R)
\\
&\oplus
\bigoplus_{P\in\Gamma\backslash\binom X2}
\left(
R^{|\Gamma\mathbin{\cdot}P|-1}
\oplus
Z^1(\Gamma_P,R_{\varepsilon_P})
\right).
\end{aligned}
\end{equation}
\end{corollary}

\begin{proof}
Shapiro's lemma and the point-orbit decomposition give
\[
H^2(\Gamma,R[X])
\cong
\bigoplus_{x\in\Gamma\backslash X}H^2(\Gamma_x,R).
\]
To describe the cocycle-level Shapiro isomorphism, let \(H\leq\Gamma\)
and let \(M\) be an \(H\)-module. Realize the coinduced module as the module of functions
\(F:\Gamma\to M\) satisfying \(F(hg)=h\cdot F(g)\), with action
\((g\cdot F)(y)=F(yg)\). Choose a transversal \(T\) for
\(H\backslash\Gamma\), with \(1\in T\), and write \(y=h_yt_y\), where
\(h_y\in H\) and \(t_y\in T\). For \(t\in T\) and \(g\in\Gamma\),
write
\[
tg=\kappa(t,g)(t\star g),
\qquad
\kappa(t,g)\in H,\quad t\star g\in T.
\]
The Shapiro map
\[
\operatorname{sh}(c)(h)=c(h)(1)
\]
is split surjective on cocycles. Indeed, for \(b\in Z^1(H,M)\), a
right inverse sends \(b\) to the cocycle \(c_b\) defined by
\[
c_b(g)(y)
=
h_y\cdot b\bigl(\kappa(t_y,g)\bigr),
\]
as follows from
\(\kappa(t,gg')=\kappa(t,g)\kappa(t\star g,g')\).

If \(c\in\ker(\operatorname{sh})\), put \(F_c(y)=c(y)(1)\). Then
\(F_c(hy)=h\cdot F_c(y)\), \(F_c(1)=0\), and
\[
c(g)(y)=F_c(yg)-F_c(y)=(g\cdot F_c-F_c)(y).
\]
Conversely, every \(F\in\operatorname{Coind}_H^\Gamma M\) with
\(F(1)=0\) gives a kernel element. This parametrization is injective.
Indeed, if \(g\cdot F-F=0\), then \(F\) is \(\Gamma\)-invariant and hence
\(F(y)=F(1)=0\) for every \(y\in\Gamma\). Restriction to
\(T\setminus\{1\}\) therefore identifies these functions with
\(M^{|T|-1}\). Since induction and coinduction agree at finite index,
this gives the transversal-dependent
isomorphism~\cite[Chapter~III]{Brown1982}
\[
Z^1(\Gamma,\Ind_H^\Gamma M)
\cong
M^{[\Gamma:H]-1}\oplus Z^1(H,M).
\]
Applying this identity to the pair-orbit decomposition in
Proposition~\ref{prop:perm-decomposition}, and then using
Theorem~\ref{thm:permutation-kac-vanishing}, proves
\eqref{eq:full-opext-orbit-formula}. Its noncanonicity comes from the
identification \(A^\vee\simeq R[X]\), the choices of orientations and
transversals, and the orbit representatives.
\end{proof}

For the natural action of \(\mathfrak S_m\) on
\(X=\{1,\dots,m\}\), assume \(m\geq2\), so that \(R[X]=R^m\). There is
one orbit on points and one orbit on unordered pairs. Their stabilizers
are \(\mathfrak S_{m-1}\) and
\(\mathfrak S_2\times\mathfrak S_{m-2}\), respectively, and the pair
orbit has \(\binom{m}{2}\) elements. The cocycle calculation in the
proof of Proposition~\ref{prop:symmetric-alt-cohomology} gives
\[
Z^1(\mathfrak S_2\times\mathfrak S_{m-2},R_\varepsilon)
\cong
\begin{cases}
R,&m=2,3,\\
R\oplus R[2],&m\geq4.
\end{cases}
\]
Consequently, Corollary~\ref{cor:full-opext-orbit-formula} gives the
noncanonical isomorphism
\begin{equation}\label{eq:symmetric-full-opext}
\Opext_{\mathfrak S_m}(R^m)
\cong
H^2(\mathfrak S_{m-1},R)
\oplus
\begin{cases}
R^{\binom{m}{2}},&m=2,3,\\[1mm]
R^{\binom{m}{2}}\oplus R[2],&m\geq4.
\end{cases}
\end{equation}
The universal coefficient theorem and the Schur multiplier of the
symmetric group give the noncanonical
isomorphisms~\cite{Brown1982,Karpilovsky1987}
\[
H^2(\mathfrak S_r,R)
\cong
\begin{cases}
0,&r=0,1,\\[1mm]
R/2R,&r=2,3,\\[1mm]
R/2R\oplus R[2],&r\geq4.
\end{cases}
\]
Since \(R=\Z/n\Z\), substituting \(r=m-1\) in
\eqref{eq:symmetric-full-opext} gives the noncanonical isomorphism
\begin{equation}\label{eq:symmetric-full-opext-explicit}
\Opext_{\mathfrak S_m}(R^m)
\cong
R^{\binom{m}{2}}
\oplus
\begin{cases}
0,&n\text{ odd or }m=2,\\[1mm]
\Z/2\Z,&n\text{ even and }m=3,\\[1mm]
(\Z/2\Z)^2,&n\text{ even and }m=4,\\[1mm]
(\Z/2\Z)^3,&n\text{ even and }m\geq5.
\end{cases}
\end{equation}

Let \(m,n\geq2\), and let \(H_{n,m}\) be the generalized
Kac--Paljutkin algebra constructed in~\cite{Lomp2025} from the primitive
\(n\)-th root \(q\) fixed in Section~\ref{sec:permutation}. The root
\(q\) determines an \(\mathfrak S_m\)-equivariant identification
\[
\kk[(\Z/n\Z)^m]\simeq \kk^{(\Z/n\Z)^m}.
\]
Let
\[
\mathfrak s:
Z^1\bigl(\mathfrak S_m,\AltBich(R^m)\bigr)
\longrightarrow
\Opext_{\mathfrak S_m}(R^m)
\]
be the homomorphic section constructed in
Theorem~\ref{thm:permutation-kac-vanishing}.

\begin{proposition}
\label{prop:Hnm-extension-class}
With this notation,
\[
[H_{n,m}]=\mathfrak s(d_{\operatorname{inv}}).
\]
This class has order \(n\).
\end{proposition}

\begin{proof}
Let \(x_j\) be the standard generator of the \(j\)-th cyclic factor of
\((\Z/n\Z)^m\). The identification determined by \(q\) sends \(x_j^\ell\) to the
function \(a\mapsto q^{\ell a_j}\). Hence the defining elements of
\(H_{n,m}\),
\[
J_i
=
\frac1n\sum_{r,s=0}^{n-1}q^{-rs}x_i^r\otimes x_{i+1}^s,
\qquad
t_i=\mu(J_i),
\]
where \(\mu\) denotes multiplication, become
\[
J_i(a,b)=q^{a_i b_{i+1}}=c_i(a,b),
\qquad
t_i(a)=q^{a_i a_{i+1}}.
\]
For \(r\ne s\), write \(c_{rs}(a,b)=q^{a_rb_s}\), and let \(z_i\) be
the generator of \(H_{n,m}\) mapping to \(s_i\). For a reduced
expression \(w=s_{i_1}\cdots s_{i_\ell}\), put \(w_0=1\),
\(w_j=s_{i_1}\cdots s_{i_j}\), and
\(z_w=z_{i_1}\cdots z_{i_\ell}\). The braid and commutation relations
make \(z_w\) independent of the reduced expression. Define \(J(w)\) by
\[
\Delta(z_w)=J(w)(z_w\otimes z_w).
\]
Also define
\[
\mathcal I(w)
=
\{(p,r)\mid p<r,\ w^{-1}(p)>w^{-1}(r)\}.
\]
Then the reduced-word inversion formula gives
\[
J(w)
=
\prod_{j=1}^{\ell}w_{j-1}\cdot c_{i_j,i_j+1}
=
\prod_{(p,r)\in\mathcal I(w)}c_{pr}
=
\tau^{d_{\operatorname{inv}}}_w.
\]

Choose \(\zeta\in\kk^\times\) with \(\zeta^2=q\), which exists because
\(\kk\) is algebraically closed, and fix a normalized set-theoretic
section \(R\to\Z\), \(r\mapsto\bar r\), used in every coordinate. Put
\[
f_{rs}(a)=\zeta^{\bar a_r\bar a_s}=f_{sr}(a),
\qquad
f_w(a)
=
\prod_{(p,r)\in\mathcal I(w)}f_{pr}(a).
\]
Set \(f_i=f_{i,i+1}\) and \(u_i=f_i^{-1}z_i\). Since
\(z_i^2=t_i\), \(s_i\cdot f_i=f_i\), and \(f_i^2=t_i\), one has
\(u_i^2=1\). Expanding along a reduced expression gives
\[
z_w
=
f_w\,u_{i_1}\cdots u_{i_\ell}.
\]
The braid and commutation moves preserve both \(z_w\) and the
inversion-set expression for \(f_w\). Hence the \(u_i\) satisfy the
Coxeter relations. If \(u_w\) denotes their resulting product, then
\(z_w=f_wu_w\).

It remains to compare the multiplication cocycles. Let
\(\widetilde U_w\) be the integral upper-triangular matrix whose
\((p,r)\)-entry is the indicator of \((p,r)\in\mathcal I(w)\).
The signed inversion-set identity says that the assignment
\(w\mapsto\widetilde U_w-\widetilde U_w^{\mathsf T}\) is an integral
crossed homomorphism. Consequently,
\[
\widetilde C(g,h)
=
\widetilde U_g+g\cdot\widetilde U_h-\widetilde U_{gh}
\]
is symmetric with zero diagonal, and its reduction modulo \(n\) is
\(C_{d_{\operatorname{inv}}}(g,h)\) from
\eqref{eq:C-d-matrix}. Since \(\widetilde C(g,h)\) is symmetric, adding
the transpose of its defining identity shows that the defect of
\(\widetilde U_w+\widetilde U_w^{\mathsf T}\) is
\(2\widetilde C(g,h)\).
Therefore, using the same set-theoretic section in every coordinate,
\begin{align*}
(\partial_\Gamma f)(g,h)(a)
&=
f_g(a)(g\cdot f_h)(a)f_{gh}(a)^{-1}\\
&=
\zeta^{
2\sum_{p<r}\widetilde C(g,h)_{pr}\bar a_p\bar a_r}\\
&=
q^{
\sum_{p<r}C_{d_{\operatorname{inv}}}(g,h)_{pr}a_pa_r}\\
&=
\mathfrak q_{C_{d_{\operatorname{inv}}}(g,h)}(a)
=
\sigma^{d_{\operatorname{inv}}}_{g,h}(a).
\end{align*}
The identity for \(\widetilde C(g,h)\) produces the factor \(2\)
integrally, so no invertibility assumption on \(2\in R\) is used.
If \(\sigma^{(z)}\) denotes the multiplication cocycle associated with the section
\(w\mapsto z_w\), the identity \(z_w=f_wu_w\) gives
\[
\sigma^{(z)}(g,h)
=
f_g(g\cdot f_h)f_{gh}^{-1}
=
(\partial_\Gamma f)(g,h)
=
\sigma^{d_{\operatorname{inv}}}_{g,h}.
\]
Together with \(J(w)=\tau^{d_{\operatorname{inv}}}_w\), this shows,
with the conventions of~\eqref{eq:full-total-cocycle}, that the total
cocycle of \(H_{n,m}\) associated with this section is precisely the
one defining \(\mathfrak s(d_{\operatorname{inv}})\). The Coxeter
relations for the \(u_i\) show that \(s_i\mapsto u_i\) is an algebra
section. Moreover,
\[
\Delta(u_i)
=
(c_i\,\partial f_i)(u_i\otimes u_i).
\]
Since \(\sigma^{(z)}=\partial_\Gamma f\), it follows that
\(\operatorname{mult}([H_{n,m}])=1\).

Finally, \(\AltBich(R^m)\) has exponent \(n\), while
\(d_{\operatorname{inv}}(s_i)=E_{i,i+1}\) has order \(n\). Hence
\(d_{\operatorname{inv}}\) has order \(n\), and the injectivity of
\(\mathfrak s\) gives the same order for \([H_{n,m}]\).
\end{proof}

Consequently,
\[
(\vartheta_*^{-1}\circ\operatorname{cl})([H_{n,m}])
=d_{\operatorname{inv}},
\qquad
\operatorname{mult}([H_{n,m}])=1.
\]

The bicharacters \(c_i\) associated with the defining section
\(s_i\mapsto z_i\) do not themselves form strict Hopf data for any
\(n\geq2\). Their crossed defect is canceled by \(\partial\sigma^{(z)}\) in
the mixed identity~\eqref{eq:full-total-cocycle}. Passing to
\(u_i=f_i^{-1}z_i\) makes \(\sigma^{(z)}\) trivial and produces the strict datum
\(c_i\partial f_i\). When \(n\) is even,
Proposition~\ref{prop:inversion-lift} shows that \([H_{n,m}]\) has no
algebra-split representative whose coproduct cocycles are bicharacters
under equivalences inducing the identity on \(\kk^{R^m}\) and
\(\kk\mathfrak S_m\).

\section{Coxeter modules}\label{sec:arithmetic-coxeter}

The Coxeter examples arise from arithmetic reductions of the geometric
representation, which is naturally defined over a ring of algebraic
integers. Reduction modulo finite-index ideals produces finite
coefficient modules on which cohomological data can be studied even for
infinite or non-crystallographic Coxeter groups. We first express the
full Coxeter cohomological problem by relations on the simple
reflections and then isolate a computable linearized subfamily. Finite
dihedral groups and affine type \(\widetilde A_1\) provide concrete
calculations.

Let \((W,S)\) be a Coxeter system of finite rank, with
\(S=\{s_i\}_{i\in I}\) and Coxeter matrix \((m_{ij})\).

For \(i\ne j\), set \(c_{ij}=2\cos(\pi/m_{ij})\) when
\(m_{ij}<\infty\), and set \(c_{ij}=2\) when \(m_{ij}=\infty\).
This normalization corresponds to the standard geometric
representation after choosing simple roots whose off-diagonal Cartan
coefficients are \(-c_{ij}\); see~\cite[Chapter~5]{Humphreys1990}.

Define the Coxeter field
\[
K_W
=
\mathbb Q(c_{ij}\mid i\ne j,\ m_{ij}<\infty),
\]
and let \(\mathcal O_W\) be its ring of integers.
For \(i\ne j\) and finite \(m_{ij}\), the number \(c_{ij}\) is the sum of a
\(2m_{ij}\)-th root of unity and its inverse, so it belongs to
\(\mathcal O_W\); the same is clear for \(c_{ij}=2\).
The associated arithmetic Coxeter lattice is
\[
L_W
=
\bigoplus_{i\in I}\mathcal O_W\alpha_i.
\]

The simple reflections act by
\[
s_i(\alpha_i)=-\alpha_i,
\qquad
s_i(\alpha_j)
=
\alpha_j+c_{ij}\alpha_i
\quad (j\neq i).
\]
This yields the geometric representation
\(W\longrightarrow\Aut_{\mathcal O_W}(L_W)\).

Let \(\mathfrak a\subsetneq\mathcal O_W\) be an ideal of finite index and
set
\[
R_{\mathfrak a}
=
\mathcal O_W/\mathfrak a,
\qquad
L_{\mathfrak a}
=
L_W/\mathfrak a L_W.
\]

Then \(L_{\mathfrak a}\cong R_{\mathfrak a}^{|S|}\) is a finite
\(R_{\mathfrak a}\)-module equipped with an induced action
of \(W\). Viewed additively, \(L_{\mathfrak a}\) is the finite abelian
group to which the obstruction theory of
Section~\ref{sec:obstructions} applies. Thus every finite-rank Coxeter
system, even when \(W\) is
infinite or non-crystallographic, gives rise to a finite coefficient
module after arithmetic reduction.

The induced action of \(W\) on \(L_{\mathfrak a}\) determines an
action on \(\AltBich(L_{\mathfrak a})\), the group of alternating
bicharacters on \(L_{\mathfrak a}\). In this section we write this
coefficient group additively.

A \emph{Coxeter cohomological datum modulo \(\mathfrak a\)} is a
crossed homomorphism \(d:W\to\AltBich(L_{\mathfrak a})\).
Via \(\vartheta\), it specifies the crossed family of cohomology classes
\(\vartheta_*d\). A strict Hopf datum realizing this family exists
precisely when \(\delta(\vartheta_*d)\) is trivial, by
Theorem~\ref{thm:general-obstruction}; a bicharacter lift of \(d\) exists
precisely when \(\operatorname{ob}_{\operatorname{bich}}(d)\) is
trivial, by Theorem~\ref{thm:bich-obstruction}.

\begin{definition}
A \emph{Coxeter \(1\)-cocycle datum} is a family
\((\xi_i)_{i\in I}\) of elements of \(\AltBich(L_{\mathfrak a})\) such
that
\[
\xi_i+s_i\cdot\xi_i=0
\]
for every \(i\in I\), and, for every finite braid relation \(m_{ij}<\infty\),
\[
\xi_i+s_i\cdot\xi_j+s_i s_j\cdot\xi_i+\cdots
=
\xi_j+s_j\cdot\xi_i+s_j s_i\cdot\xi_j+\cdots,
\]
where both sides have \(m_{ij}\) terms and the indices alternate.
\end{definition}

\begin{proposition}\label{prop:coxeter-crossed}
The assignment \(d\mapsto(d(s_i))_{i\in I}\) identifies
\(Z^1(W,\AltBich(L_{\mathfrak a}))\) with the set of Coxeter
\(1\)-cocycle data.
\end{proposition}

\begin{proof}
Since \(W\) is generated by \(S\), a crossed homomorphism is determined
by the values \(\xi_i=d(s_i)\). Applying
\(d(xy)=d(x)+x\cdot d(y)\) to the involutivity and braid relations gives
the defining Coxeter \(1\)-cocycle equations. Conversely, satisfaction
of the involutivity and braid equations means precisely that, for
\(M=\AltBich(L_{\mathfrak a})\),
\(s_i\mapsto(\xi_i,s_i)\) respects the Coxeter presentation in
\(M\rtimes W\). The multiplication in this semidirect product is
\((m,g)(n,h)=(m+g\cdot n,gh)\); thus the first component of the
resulting homomorphism \(W\to M\rtimes W\) satisfies
\(d(gh)=d(g)+g\cdot d(h)\) and is the required crossed homomorphism.
\end{proof}

Proposition~\ref{prop:coxeter-crossed} identifies
\(Z^1(W,\AltBich(L_{\mathfrak a}))\) with the solution set of a finite
system of relations determined by the Coxeter graph and the induced
action on \(\AltBich(L_{\mathfrak a})\). This describes the full
Coxeter cohomological problem. For the arithmetic calculations, we
restrict to the linearized subfamily, whose
image need not exhaust \(\AltBich(L_{\mathfrak a})\).

The relevant coefficient module is
\[
M_{\mathfrak a}^{\operatorname{lin}}
=
\Hom_{R_{\mathfrak a}}
\bigl(
\wedge^2_{R_{\mathfrak a}}L_{\mathfrak a},
R_{\mathfrak a}
\bigr).
\]
Equip it with the contragredient action
\[
(g\cdot\omega)(v\wedge u)
=
\omega(g^{-1}v\wedge g^{-1}u).
\]
Choose an additive character
\(\chi\colon (R_{\mathfrak a},+)\longrightarrow \kk^\times\).
Composition with \(\chi\) defines a \(W\)-equivariant homomorphism
\[
\begin{aligned}
\iota_\chi:
M_{\mathfrak a}^{\operatorname{lin}}
&\longrightarrow
\AltBich(L_{\mathfrak a}),\\
\iota_\chi(\omega)(v,w)
&=
\chi\bigl(\omega(v\wedge w)\bigr).
\end{aligned}
\]
Call \(\chi\) \emph{generating} if its kernel contains no nonzero ideal
of \(R_{\mathfrak a}\). Then \(\iota_\chi\) is injective, since the image
of a nonzero \(R_{\mathfrak a}\)-linear form is a nonzero ideal and hence
is not contained in \(\ker\chi\).
Generating characters exist. Indeed, after factoring \(\mathfrak a\)
into prime powers, each local factor has ideals forming a chain and a
smallest nonzero ideal. One may choose a character nontrivial on that ideal
and take the product of these characters via the Chinese remainder theorem.
Fix one such character \(\chi\). Then \(\iota_\chi\) identifies the
linearized module with a \(W\)-submodule of
\(\AltBich(L_{\mathfrak a})\) and induces an injection
\[
Z^1(W,M_{\mathfrak a}^{\operatorname{lin}})
\lhook\joinrel\longrightarrow
Z^1(W,\AltBich(L_{\mathfrak a})).
\]
We use the standard notation \(Z^1(G,M)\), \(B^1(G,M)\), and
\(H^1(G,M)\). The group \(Z^1(G,M)\) parametrizes the actual crossed
data, whereas \(H^1(G,M)\) records them modulo principal crossed
homomorphisms. The cohomology groups considered here are those of the
linearized module itself; the lifting arguments use only the injection
of \(Z^1\)-groups. Every lifting statement for
\(d\in Z^1(W,M_{\mathfrak a}^{\operatorname{lin}})\) is applied to
the alternating bicharacter datum \(\iota_\chi\circ d\).

\subsection{Finite Coxeter groups}\label{sec:finite-coxeter}

\begin{proposition}\label{prop:finite-averaging}
If \(W\) is finite and \(|W|\) is invertible in
\(R_{\mathfrak a}\), then
\[
H^1(W,M_{\mathfrak a}^{\operatorname{lin}})=0.
\]
\end{proposition}

\begin{proof}
Let \(d:W\to M_{\mathfrak a}^{\operatorname{lin}}\) be a crossed
homomorphism. Since \(|W|\) is invertible in
\(R_{\mathfrak a}\), define
\[
m
=
-|W|^{-1}
\sum_{w\in W} d(w).
\]
Since \(w\mapsto gw\) permutes \(W\), summing
\(d(gw)=d(g)+g\cdot d(w)\) over \(w\in W\) gives
\(\sum_w d(w)=|W|d(g)+g\cdot\sum_w d(w)\). Multiplication by
\(-|W|^{-1}\) therefore gives
\(d(g)=g\cdot m-m\), so \(d\) is principal.
\end{proof}

After decomposing \(R_{\mathfrak a}\) into primary factors, non-principal
linear classes can therefore occur only at primes of
\(R_{\mathfrak a}\) lying over rational primes dividing \(|W|\).

The finite dihedral case admits an explicit computation.

For an integer \(m\ge2\), write
\[
I_2(m)
=
\langle
s,t
\mid
s^2=t^2=1,\,
(st)^m=1
\rangle.
\]

For \(W=I_2(m)\), the arithmetic reduction \(L_{\mathfrak a}\) of the
Coxeter lattice is free of rank two over \(R_{\mathfrak a}\), and its
linear coefficient module is free of rank one.
Let \(R_{\mathfrak a}^{\operatorname{sgn}}\) denote the free rank-one
\(R_{\mathfrak a}\)-module on which each simple reflection acts by
\(-1\).

\begin{proposition}\label{prop:dihedral-sign}
For \(W=I_2(m)\),
\[
M_{\mathfrak a}^{\operatorname{lin}}
\cong
R_{\mathfrak a}^{\operatorname{sgn}}.
\]
\end{proposition}

\begin{proof}
Let \(\omega_0\in M_{\mathfrak a}^{\operatorname{lin}}\) be defined by
\(\omega_0(\alpha_1\wedge\alpha_2)=1\). The map \(r\mapsto r\omega_0\) identifies
\(M_{\mathfrak a}^{\operatorname{lin}}\) with \(R_{\mathfrak a}\).
An element \(w\in W\) acts on the dual exterior square by
\(\det(w)^{-1}\). Each simple reflection has determinant \(-1\), and
the simple reflections generate \(W\), so the resulting action is the
sign character.
\end{proof}

Write \(R_{\mathfrak a}[m]=\{x\in R_{\mathfrak a}\mid mx=0\}\).

\begin{theorem}\label{thm:dihedral-cohomology}
For the finite dihedral group \(I_2(m)\),
\[
\begin{aligned}
H^1(I_2(m),M_{\mathfrak a}^{\operatorname{lin}})
&\cong H^1(I_2(m),R_{\mathfrak a}^{\operatorname{sgn}})\\
&\cong R_{\mathfrak a}/2R_{\mathfrak a}
\oplus R_{\mathfrak a}[m].
\end{aligned}
\]
\end{theorem}

\begin{proof}
Write \(a=d(s)\) and \(b=d(t)\). Since both generators act by \(-1\)
on the sign module, their involutivity relations impose no conditions.
Indeed, \(d(s^2)=a+s\cdot a=a-a=0\), and similarly \(d(t^2)=0\). Moreover,
\(d(st)=a+s\cdot b=a-b\). Since \(st\) acts trivially, the relation
\((st)^m=1\) gives
\[
0=d((st)^m)=m(a-b).
\]

Thus
\[
Z^1(I_2(m),R_{\mathfrak a}^{\operatorname{sgn}})
=
\{
(a,b)\in R_{\mathfrak a}^2
\mid
m(a-b)=0
\}.
\]

A principal crossed homomorphism has the form \(d_u(g)=g\cdot u-u\),
hence \((d_u(s),d_u(t))=(-2u,-2u)\).

Therefore
\[
B^1(I_2(m),R_{\mathfrak a}^{\operatorname{sgn}})
=
\{
(-2u,-2u)
\mid
u\in R_{\mathfrak a}
\}.
\]

Under the identification
\(Z^1(I_2(m),R_{\mathfrak a}^{\operatorname{sgn}})\subseteq
R_{\mathfrak a}^2\), define
\[
\begin{aligned}
\Phi:Z^1(I_2(m),R_{\mathfrak a}^{\operatorname{sgn}})
&\longrightarrow
R_{\mathfrak a}/2R_{\mathfrak a}\oplus R_{\mathfrak a}[m],\\
(a,b)&\longmapsto
\bigl(b\bmod 2R_{\mathfrak a},\,a-b\bigr).
\end{aligned}
\]
It is well defined because \(m(a-b)=0\). It is
surjective. Indeed, if
\(c\in R_{\mathfrak a}\) and \(r\in R_{\mathfrak a}[m]\), then
\((c+r,c)\) maps to \((c\bmod 2R_{\mathfrak a},r)\). Its kernel is
formed by the pairs \((a,b)\) with \(a=b\) and
\(b\in 2R_{\mathfrak a}\), hence it is exactly
\(B^1(I_2(m),R_{\mathfrak a}^{\operatorname{sgn}})\). Therefore
the quotient by \(B^1\) is
\(R_{\mathfrak a}/2R_{\mathfrak a}\oplus R_{\mathfrak a}[m]\).
\end{proof}

\begin{remark}
The two summands reflect distinct sources of cohomology, namely failure
of divisibility by \(2\) and \(m\)-torsion arising from the braid
relation.
\end{remark}

When \(2\) is invertible in \(R_{\mathfrak a}\), the additive exponent
of \(L_{\mathfrak a}\) is odd. For the fixed generating character
\(\chi\), Corollary~\ref{cor:odd-exponent} provides a canonical
bicharacter lift of \(\iota_\chi\circ d\) for every
\(d\in Z^1(I_2(m),M_{\mathfrak a}^{\operatorname{lin}})\).
Consequently, each such lifted datum produces a Hopf algebra
of dimension \(|I_2(m)|\,|L_{\mathfrak a}|=2m\,|R_{\mathfrak a}|^2\).

\begin{example}\label{ex:A2-arithmetic-module}
Here \(A_2=I_2(3)\), \(c_{12}=c_{21}=1\), and
\(\mathcal O_W=\Z\). For \(R=\Z/n\Z\) with \(n\ge2\), the reduced module
\(L_n=R\alpha_1\oplus R\alpha_2\cong R^2\) has action
\[
\begin{aligned}
s_1(\alpha_1)&=-\alpha_1,
&
s_1(\alpha_2)&=\alpha_2+\alpha_1,\\
s_2(\alpha_1)&=\alpha_1+\alpha_2,
&
s_2(\alpha_2)&=-\alpha_2.
\end{aligned}
\]
Theorem~\ref{thm:dihedral-cohomology} gives
\[
H^1(I_2(3),R^{\operatorname{sgn}})
\cong
R/2R\oplus R[3].
\]
Thus nonzero linearized cohomology classes occur exactly when
\(2\mid n\) or \(3\mid n\).
\end{example}

\begin{example}
For type \(I_2(5)\), the arithmetic ring is
\(\mathcal O_W=\Z[\varphi]\), where
\(\varphi^2-\varphi-1=0\).

Hence
\[
H^1(I_2(5),R_{\mathfrak a}^{\operatorname{sgn}})
\cong
R_{\mathfrak a}/2R_{\mathfrak a}
\oplus
R_{\mathfrak a}[5].
\]

The \(R_{\mathfrak a}[5]\)-summand shows that quotients supported at
the ramified prime lying over \(5\) produce nonzero linearized
arithmetic cohomology classes.
\end{example}

\subsection{Affine type \texorpdfstring{$\widetilde A_1$}{A1 tilde}}
\label{sec:affine-coxeter}

Since \(\widetilde A_1\) is infinite, the averaging argument of
Proposition~\ref{prop:finite-averaging} does not apply. To obtain a
finite-dimensional Hopf algebra from this infinite Coxeter group, we
pass to a finite quotient that retains both the reduced action and the
lifted datum.
Proposition~\ref{prop:affine-factorization} supplies this simultaneous
factorization for any bicharacter-normalized datum.

For \(g\in\Aut(L_{\mathfrak a})\) and
\(c\in\Bich(L_{\mathfrak a})\), put
\((g\cdot c)(v,u)=c(g^{-1}v,g^{-1}u)\).
For a bicharacter-normalized Hopf datum
\(b\in Z^1(W,\Bich(L_{\mathfrak a}))\), set
\[
W
\xrightarrow{\Phi_b}
\Bich(L_{\mathfrak a})
\rtimes
\Aut(L_{\mathfrak a}),
\qquad
\Phi_b(w)=(b_w,w|_{L_{\mathfrak a}}).
\]
Let \(\Gamma_{\mathfrak a,b}\) be the image of \(\Phi_b\).

\begin{proposition}\label{prop:affine-factorization}
Let \(W\) be any group acting on the finite abelian group
\(L_{\mathfrak a}\), and let \(b\) be a bicharacter-normalized Hopf
datum. Then \(\Gamma_{\mathfrak a,b}\) is finite, and both the action on
\(L_{\mathfrak a}\) and \(b\) descend to
\(\Gamma_{\mathfrak a,b}\). Writing \(\bar b\) for the descended datum,
it determines a finite-dimensional Hopf algebra
\(H(\Gamma_{\mathfrak a,b},L_{\mathfrak a},\bar b)\).
\end{proposition}

\begin{proof}
The crossed homomorphism identity for \(b\) shows that \(\Phi_b\) is a
homomorphism. Let \(N\) be the exponent of \(L_{\mathfrak a}\). Since
\(L_{\mathfrak a}\) is finite, \(\Aut(L_{\mathfrak a})\) is finite, and
every bicharacter \(c\) satisfies
\(c(v,w)^N=c(Nv,w)=1\), so
\(\Bich(L_{\mathfrak a})\subseteq
\mu_N(\kk)^{L_{\mathfrak a}\times L_{\mathfrak a}}\) is finite. Thus
\(\Gamma_{\mathfrak a,b}\) is finite. Let \((c,g)\) act on
\(L_{\mathfrak a}\) through \(g\), and set \(\bar b_{(c,g)}=c\). The
semidirect-product law makes \(\bar b\) a crossed homomorphism, and
\(\bar b_{\Phi_b(w)}=b_w\), so both the action and \(b\) descend. Since
\(\Gamma_{\mathfrak a,b}\) and \(L_{\mathfrak a}\) are finite, the
associated Hopf algebra is finite-dimensional.
\end{proof}

The smallest affine Coxeter group is
\[
\widetilde A_1
=
I_2(\infty)
=
\langle s,t\mid s^2=t^2=1\rangle.
\]

For \(\widetilde A_1\), the Coxeter field is \(\mathbb Q\), so
\(\mathcal O_W=\Z\).
For \(n\ge2\), take \(\mathfrak a=n\Z\), write
\(R=R_{\mathfrak a}=\Z/n\Z\), and set
\(L_n=L_{\mathfrak a}=R\alpha_1\oplus R\alpha_2\cong R^2\). Fix a
generating additive character \(\chi\colon(R,+)\to\kk^\times\).
The generators act by
\[
\begin{aligned}
s(\alpha_1)&=-\alpha_1,
&
s(\alpha_2)&=\alpha_2+2\alpha_1,\\
t(\alpha_1)&=\alpha_1+2\alpha_2,
&
t(\alpha_2)&=-\alpha_2.
\end{aligned}
\]

Writing
\(M_n^{\operatorname{lin}}=
\Hom_R(\wedge_R^2L_n,R)\), the linear alternating module is free of
rank one over \(R\). The form \(\omega_0\) defined by
\(\omega_0(\alpha_1\wedge\alpha_2)=1\) identifies it with \(R\), and both
generators act by \(-1\). Thus
\(M_n^{\operatorname{lin}}\cong R^{\operatorname{sgn}}\) through this
fixed identification.

The involutivity relations impose no conditions on the values at \(s\)
and \(t\), and there are no braid relations. Set
\(\operatorname{diag}(2R)=\{(2u,2u)\mid u\in R\}\). Evaluation at
\(s\) and \(t\), together with
\(d_u(s)=d_u(t)=-2u\) for a principal crossed homomorphism, gives
\[
\begin{aligned}
Z^1(\widetilde A_1,M_n^{\operatorname{lin}})&\cong R^2,\\
B^1(\widetilde A_1,M_n^{\operatorname{lin}})&=\operatorname{diag}(2R).
\end{aligned}
\]

\begin{proposition}\label{prop:affine-a1}
For every \(n\ge2\),
\[
\begin{aligned}
H^1(\widetilde A_1,M_n^{\operatorname{lin}})
&\cong H^1(\widetilde A_1,R^{\operatorname{sgn}})\\
&\cong R\oplus R/2R.
\end{aligned}
\]
In particular, if \(n\) is odd, then
\[
H^1(\widetilde A_1,M_n^{\operatorname{lin}})
\cong R.
\]
\end{proposition}

\begin{proof}
Consider the homomorphism
\[
\Psi:R^2\longrightarrow R\oplus R/2R,
\qquad
(a,b)\longmapsto(a-b,\,b\bmod 2R).
\]
Given \((r,\bar c)\in R\oplus R/2R\), choose a representative
\(c\in R\) of \(\bar c\); then \(\Psi(r+c,c)=(r,\bar c)\), so \(\Psi\)
is surjective. If \(\Psi(a,b)=0\), then \(a=b\) and \(b\in2R\), and
conversely every pair \((2u,2u)\) maps to zero. Hence
\(\ker\Psi=\operatorname{diag}(2R)\), which is the group of principal
crossed homomorphisms. Thus \(\Psi\) induces the claimed isomorphism.
If \(n\) is odd, then \(2R=R\), so \(R/2R=0\).
\end{proof}

For \(d\in Z^1(\widetilde A_1,M_n^{\operatorname{lin}})\), set
\(d_\chi=\iota_\chi\circ d\). If \(n\) is odd, the fixed generating
character \(\chi\) and Corollary~\ref{cor:odd-exponent} give a canonical
bicharacter lift of \(d_\chi\).
Proposition~\ref{prop:affine-factorization} then produces a
finite-dimensional Hopf algebra after passing to the finite image
determined by the reduced action and the lifted datum. If \(n\) is even,
the datum \(\vartheta_*d_\chi\) admits a strict lift precisely when
\(\delta(\vartheta_*d_\chi)\) is trivial, by
Theorem~\ref{thm:general-obstruction}; the datum \(d_\chi\) admits a
bicharacter lift precisely when
\(\operatorname{ob}_{\operatorname{bich}}(d_\chi)\) is trivial, by
Theorem~\ref{thm:bich-obstruction}.

Theorem~\ref{thm:dihedral-cohomology} shows that finite dihedral
linearized cohomology is supported at primes lying over rational primes
dividing \(2m\), whereas Proposition~\ref{prop:affine-a1} exhibits an
\(R\)-summand for \(\widetilde A_1\) when \(n\) is odd. This is the
finite--affine contrast detected by the linearized Coxeter calculations.

\bibliographystyle{amsalpha}
\bibliography{references}

\end{document}